\documentclass[letterpaper,reqno,twoside]{amsart}

\usepackage{amsthm,amssymb,dsfont,mathtools}
\usepackage[initials]{amsrefs}
\usepackage[american]{babel} 
\usepackage{enumerate}
\usepackage[T1]{fontenc}
\usepackage{ae,aecompl} 

\newtheorem{theorem}{Theorem}[section]
\newtheorem{lemma}[theorem]{Lemma}
\newtheorem{corollary}[theorem]{Corollary}
\newtheorem{proposition}[theorem]{Proposition}

\theoremstyle{definition}
\newtheorem{definition}[theorem]{Definition}

\theoremstyle{remark}
\newtheorem{remark}{Remark}[section]

\numberwithin{equation}{section}

\newcommand{\ds}{\displaystyle}
\newcommand{\R}{\mathbb{R}}
\newcommand{\N}{\mathbb{R}}
\newcommand{\Sp}{\mathcal S}
\newcommand{\rimh}{r_{i-1/2}}
\newcommand{\riph}{r_{i+1/2}}

\newcommand{\imh}{i-1/2}

\newcommand{\pjmh}{p_{j-1/2}}
\newcommand{\pjph}{p_{j+1/2}}

\newcommand{\fh}{f^n_{j,i}}
\newcommand{\fhim}{f^n_{j,i-1}}

\newcommand{\V}{\mathcal{V}}

\newcommand{\indic}[1]{ \mathop{\mathbf{1}_{#1}} }
\newcommand{\Dr}{\Delta r}
\newcommand{\Dp}{\Delta p}
\newcommand{\Dt}{\Delta t}
\newcommand{\vphi}{\varphi}
\newcommand{\abs}[1]{\left\lvert#1\right\rvert}
\newcommand{\norm}[1]{\left\lVert#1\right\rVert}

\begin{document}

\title[A sorpion-coagulation equation]{Derivation and mathematical study of a sorption-coagulation equation}

\author[E. Hingant]{Erwan Hingant}
\address{CI\textsuperscript{2}MA -- Universidad de Concepci\'on \\ Casilla 160-C, Concepci\'on, Chile.}
\email{ehingant@ci2ma.udec.cl}

\author[M. Sep\'ulveda]{Mauricio Sep\'ulveda}
\address{CI\textsuperscript{2}MA \& DIM -- Universidad de Concepci\'on \\ Casilla 160-C, Concepci\'on, Chile.}
\email{mauricio@ing-mat.udec.cl}

\subjclass[2010]{Primary: 65R20, 82C05; Secondary: 35Q82, 35Q92}

\keywords{Polymers; Metal ions; Coagulation equation; Existence of solutions; Finite volume scheme; Weak stability}

\date{\today}

\begin{abstract}
This work is devoted to the derivation and the matematical  study of a new model for  water-soluble polymers and metal ions interactions, which are used in chemistry for their wide range of applications. 
First, we motivate and derive a model that describes the evolution of the configurational distribution of polymers. One of the novelty resides in the configuration variables which consider both, the size of the polymers and the quantity of metal ions they captured through sorption. The model consists in a non-linear transport equation with a quadratic source term, the coagulation.
Then, we prove the existence of solutions for all time to the problem thanks to classical fixed point theory. Next, we reformulate the coagulation operator under a conservative form which allows to write a finite volume scheme. The sequence of approximated solutions is proved to be convergent (toward a solution to the problem) thanks to a $L^1 - weak$ stability principle. Finally, we illustrate the behaviour of the solutions using this numerical scheme and we intend to discuss on the long-time behaviour.
\end{abstract}  

\maketitle

{ \small \tableofcontents }

\section{Introduction}

\subsection{Motivations}

A class of macromolecules, the polymers, has emerged for their potential applications in various fields, as superconducting materials, liquid crystal, and biocompatible polymers. Also, in environmental science, they can be used for instance to remove pollutant from aqueous solutions, or bacteria, fungi and algae \cites{Rivas2011,Rivas2003}. We are particularly interested in polymers which perform interactions with metal ions such as copper ions (Cu$^{2+}$), lead ions (Pb$^{2+}$), and many others.

One of the applications of this concept, being very promising, lies in membrane separation process. This technique is used to obtain highly purified water in industrial process, or to the contrary, to ``wash'' water after an industrial process before release it in the environment \cites{Zagorodni2007,Rivas2003}. The idea is to take advantage the interactions of particular polymers with metal ions in order to retain free metal ions from an aqueous solution (water). The solution is then filtrate through a membrane, and due to the high molecular weight of the polymers, they cannot cross the membrane. It results that, on one side, ions are retained by the polymers, while on the other side, water has crossed the membrane and is free of ions. We refer to \cite{Rivas2006} and the reviews \cites{Rivas2003,Rivas2011} for a more precise description of this process. 

Nevertheless, the development of this technique leads to some technical difficulties and gives rise to important questions in order to produce efficient methods. Among these, they are the role of fouling effect (adhesion of polymers to the membrane), aggregation phenomenon, concentration effect, interaction with the wall of the cell (recipient) and interaction with the fluid. For recent findings, techniques and models on the subject we refer to the works in \cites{Palencia2010,Palencia2011,Palencia2009,Palencia2009b,Moreno2002,Moreno2006,Rivas2004}. 

To go further in this direction of a better understanding of such processes, we decided to propose a new model that accounts for the polymers - metal ions and polymers - polymers interactions. The model we develop here differs from the one derived for instance in \cites{Moreno2006,Moreno2002,Rivas2006} where they used macroscopic quantities and equilibrium theory. Indeed, here we propose a new approach, ever used in other physical problems such as Ostwald ripening or coagulation-fragmentation equation, but, up to our knowledge, has never been formalised in this way for water-polymers in interactions  with metal ions. The particularity of this problem reveals mathematical difficulties as for the existence and uniqueness of solution, the numerical approximation, and the long-time behaviour of the solution. In this work, we start to answer to some of these questions and give some hints on the behaviour of the solution.  

The remainder of this introduction set the theory and derive the model studied in the present work. 

%
%

\subsection{Theory and model}

A polymer is a macromolecule made up of many repeat units called monomers. To illustrate this definition, an example of polymer is the poly(acrylic acid), a synthetic polymers. Indeed, the acrylic acid which has the molecular formula $C_3H_4O_2$ plays the role of the monomer, whereas the poly(acrylic acid) is a succession of acrylic acid, forming a chain, and has the molecular formula $-(C_3H_4O_2)_n-$, where $n$ stands for the number of occurrences of the monomer. 

The model we investigate below concerns highly soluble polymers, particularly water-soluble polymers. Rather, it aims at modelling the interactions between these polymers and metal ions in solution. These interactions are responsible for the retention of metal ions by the polymer and take place in specific sites of the chain, called functional groups. The functional groups are  repeat subunits of the polymer able to hold one metal ion each. It is a group of atoms are molecules and it can be, or the monomer itself, or else composed of several monomers. Its composition fluctuates according to: the structure of  the chain; the metal ions used; as well as the type of interaction committed. The reader can refer to \cites{Rivas2003,Rivas2006} to get some examples of such polymers, but, among these we mention  polyelectrolytes and  polychelatogens. The former is a class of polymers whose functional groups are charged, so  metal ions bind to them by an electrostatic interaction. For instance, the monomer of the poly(
acrylic acid) polymer above-mentioned, loose a proton ($H$) in solution leading to a negatively charged polymer. While for the latter metal ions bind to  functional groups forming coordinate bonds. However, some polymers combine both interactions,  even sometimes others weakest interactions can take place, see \cite{Rivas2011}. Furthermore, all depends on both the polymer and the metal ion considered.  Nevertheless, here we follow \cite{Rivas2006} and  we assume the interactions modelled might be formulated as reversible chemical reactions. Thus, we acknowledge two processes here: the binding process which consists in the association between a metal ion and a functional group of a polymer also named adsorption; and the opposite process which consists in the dissociation of a metal ion from a functional group or desorption. Both processes formed the sorption phenomenon. 

Modeling such interactions, between water-soluble polymers and metal ions, recently received a considerable attention. Several models has been developed, analysed and compared to experimental data, we would mention again the works made in \cites{Moreno2006,Moreno2002,Rivas2006}. The authors mainly deal with macroscopic quantities, which are the concentrations of metal ions bound to polymers and free, the concentration of polymers, \emph{etc}, together with equilibrium theory to build suitable models.

In the present work, we derive a model that accounts for the evolution of a system of interacting particles consisting of water-soluble polymers and metal ions in solution. Its establishment resides in a chemical formulation of both microscopic processes (association and dissociation). For this, let us label by $P_{x,y}$ a polymer consisting of $x$ functional groups, such that $y$ of them are associated to a metal ion. The variable $x\in \N^*$ can reach any numbers (at least theoretically), indeed the polymers can be as big as the technology permits it. Then, the number $y$ belongs to $\{0,\ldots,x\}$ by definition of a functional groups, since each functional group reacts with one and only one metal ion, thus a polymer made up of $x$ functional groups cannot bind more than $x$ metal ions. Remark, it makes clear that if $y$ is the number of functional groups occupied by a metal ion, it is also the number of metal ion bound to the polymer. The couple $(x,y)$ is named the configuration of the polymer. Finally, 
we 
label by $M$ a free metal ions, free meaning that it is dissolved and not bound to a polymer.

The reversible chemical reaction between one free metal ion $M$ and one polymer $P_{x,y}$ is
\begin{equation} \label{eq:kinetic}
 P_{x,y} + M \xrightleftharpoons[l_{x,y+1}]{k_{x,y}} P_{x,y+1},
\end{equation}
where $k_{x,y}$ is the binding rate (or association rate) at which a polymer of configuration $(x,y)$ interacts with a metal ions to get a new configuration $(x,y+1)$, while $l_{x,y+1}$ is the dissociation rate at which a metal ions is removed  from a polymer $(x,y+1)$, and this latter gets the new configuration $(x,y)$. Both rates depend on the configuration of the polymer and inherently on the surrounding conditions (pH of the solution, temperature, \emph{etc}) supposed to be fixed.

In addition to polymer - metal ion interactions, we take into account polymer -polymer interactions. Actually, under experimental conditions, the presence of other polymers leads to inter-polymer complexes and can produce gels or precipitates, as pointed out in \cite{Rivas2003}. We incorporate in the model the formation of inter-polymer complexes through a binary coagulation, which is the formation of a biggest polymer from two smaller ones. This reaction can be written as follows:
\begin{equation} \label{eq:coagulation}
 P_{x,y} + P_{x',y'} \xrightarrow{a_{x,y;x',y'}} P_{x+x',y+y'},
\end{equation}
where the coagulation rate $a_{x,y;x',y'}$ is the rate at which a polymer with the configuration $(x,y)$ and an other one with configuration $(x',y')$ will produce a new biggest macromolecule.  The reaction preserves the number of functional groups and metal ions bound, so the new polymers gets the configuration $(x+x',y+y')$. \\

\medskip

We devote the next section to establish the set of equations studied in the remainder of this paper. To this end we would finish here by some comments on the mathematical objects used.   

First, we use evolution equation which means solutions to the system will be time-dependent functions and are not at the equilibrium yet.

Second, the experiments use generally an high number of particles (both polymers and metal ions). That is why quantities are expressed in term of their molar or mass concentration.  So that, we introduce two new convenient variables, $p = \varepsilon x$ and $q=\varepsilon y$, respectively for the mole of functional groups and the mole of associated (to metal ions) functional groups, where $\varepsilon = 1/ \mathcal N_a << 1$ the inverse of the Avogadro's number. Both variables $p$ and $q$ can reach a continuum of values, so we consider them as continuous variables on the contrary to variables $x$ and $y$ which are natural number.   
Thus, the states of the system at a time $t\geq 0$ will be  given by a configurational distribution function $f$ where the quantity
\[\int_{\mathcal B} f(t,p,q) \, dp dq\,\]
%
%
provides the molar concentration of polymers in the system having a configuration in the subset $\mathcal B$ of the configurational space (defined latter).  

\begin{remark}
 It could be possible to keep a discrete version of the problem considering the concentration of all the $P_{x,y}$ with $x$ and $y$ the number of functional groups (as natural number).  Nevertheless, it is sometimes mathematically less tractable due to the huge number of equations involved. 
%
%
 But, the continuous model can be seen as a limit of the discrete one, through an appropriate rescaling and here we have in mind the parameters  $\varepsilon$. The reader can make himself an idea, on what would be the discrete model here and how it is linked to the continous one, in \cites{Laurencot2002c,Vasseur2002} where rigorous derivations of the continuous Lifshitz-Slyozov equation from the Becker-D\"oring model is made. Both model being closed to the one presented here.    
\end{remark}

\subsection{Equations} 

Below, we use the two continuous variables for the polymer configuration. First, $p\in \R_+ \coloneqq (0,+\infty)$ the quantity in mole of functional groups. Second, $q\in\R_+$ the quantity in mole of functional groups associate each to a metal ion. We name, the variable $q$: the quantity of occupied functional groups (in mole).  Thus, the set of admissible configurations $(p,q)$ is: 
\[S :=  \left\{ (p,q) \in \R^2_+ :  0 < q < p \right\}\,. \]
Indeed, the polymer can reach any size, so its number of functional groups can be as large as we want. While its number of occupied functional groups can belong to $(0,p)$, \emph{i.e} up to the total number of functional groups of the polymer.   Then, we define the configurational distribution function of polymers, denoted by $f(t,p,q)$, as a function of time $t\geq0$ over the configuration space $S$.  The system governing the evolution of $f$ is 
\begin{equation}\label{eq:polymer}
 \frac{\partial f}{\partial t}  + \frac{\partial}{\partial q}  \left( \V f \right) =  Q(f,f), \qquad \text{on } \ \R_+ \times S\,,
\end{equation}
where $Q$ is the coagulation operator and $\V$ denotes the rate of association-dissociation or by other means the sorption rate. The latter determines the mechanism of ions transfer with a polymer (association and dissociation of ions), it is the continuous operator describing reaction \eqref{eq:kinetic}, while the former rules the polymer - polymer interaction given by \eqref{eq:coagulation}, where both are defined below. Equation \eqref{eq:polymer} on the configurational distribution function is not enough to characterised all the system. To complete the model, it remains to introduce a second equation on the molar concentration of free metal-ions (being in the solution but unbound to polymers). We denote this concentration by $u(t)$, as a function of time $t\geq0$, and it is given by the constraint of metal ions conservation (bound and free), namely
\begin{equation}\label{eq:constraint}
u(t) + \int_S q f(t,p,q) dq dp = \rho, \qquad \text{on } \ \R_+\,,
\end{equation}
where $\rho>0$ is a constant which stands for the total quantity of metal ions in the system (bound and unbound to  polymers). Indeed, as $f(t,p,q) \Delta q \Delta p$ approaches the molar concentration of polymers with configuration $(p,q)$ in the limit of $\Delta p$ and $\Delta q$ small,  if we multiply this quantity by $q$ the number of occupied functional groups (= the number of metal ions bounded) in mole, thus we get the molar concentration of metal ions bounded onto the polymers with such a configuration. Thus, the integral term in \eqref{eq:constraint} account for the molar concentration of metal ions associated to polymers and the balance equation \eqref{eq:constraint} expresses the conservation of matter, between associate and free metal ions in the system. 

Let us now define more precisely the rate $\V$ of metal ions association-dissociation. As a general form for $\V$ we consider the following chemical reaction rate
\[\V(u(t),p,q) = k(p,q) u(t)^\gamma - l(p,q)\,,\]
where $k$ is the association rate, or adsorption, at which a free monomers bind to a polymer (depending on the type of interaction and the diffusion rate of the particles) and $l$ is the dissociation rate, or desorption (depending on the strength of the interaction). The association rate is multiply by $u(t)^\gamma$, as a classical law of mass action and for sake of simplicity we restrict ourselves to $\gamma=1$ which is the order of the reaction. A relevant example of association-dissociation rate would be an analogy to the Langmuir's law (for the adsorption of metal ions onto a surface), namely
\begin{equation} \label{eq:example_V}
 \V(u(t),p,q) = k_0 (p-q)^\alpha u(t) - l_0 q^\beta \,,
\end{equation}
with $k_0,\, l_0>0$ parameters and  $\alpha,\, \beta>0$ geometrical factor, see \cite{Somorjai2010}. Indeed, the association rate depends on the quantity of available functional groups $p-q$ while the dissociation rate depends only on the quantity of metal ions bound to the polymers $q$ (the more metal ions, the more the probability of a dissociation is great).  

Next we explicit the coagulation operator $Q$ with the coagulation rate $a$ defined as a nonnegative function over $S \times S$ satisfying the symmetry assumption
\begin{equation} \label{eq:a_symmetric}
a(p,q;p',q') = a(p',q';p,q)\,.
\end{equation}
It gives the rate at which two polymers with the configuration $(p,q)$ and $(p',q')$ will coagulate. The symmetry assumption follows form the impossibility in the system to distinguish the coagulation of a $(p,q)$ with a $(p',q')$ or the coagulation of a $(p',q')$ with a $(p,q)$ because it is the same reaction. Then, $Q$ can be decomposed by a gain term $Q^+$ and a depletion term $Q^-$  that is
\begin{equation*} 
  Q = Q^+ -Q^-,
\end{equation*}
where
\begin{align}
   Q^+(f,f)(p,q) & =   \frac 1 2 \int_0^p \int_{0}^{p'} a(p',q',p-p',q-q') \indic{(0,p-p')}(q-q') \nonumber \\ 
  &  \phantom{= \frac 1 2 \int_0^p \int_{0}^{p'} a(p',q',p-p',} \times \, f(p',q')f(p-p',q-q')\, d q' d  p'\,, \label{eq:Qplus}\\
   Q^-(f,f)(p,q) & = L(f)(p,q)  f(p,q) \,,\nonumber \\
 &  \qquad \text{with} \quad  L(f)(p,q) =  \int_0^\infty \int_0^{p'} a(p,q,p',q')f(p',q') \, d q' d p'\,. \label{eq:Qminus}
\end{align} 
The gain term $Q^+$ accounts for the production of a polymer with a configuration $(p,q)$ thanks to the coagulation of a $(p',q')$ with $q'<p'<p$ and a  $(p'-p,q'-q)$ with $0<q-q'<p-p'$. Likewise, the depletion term account for the disappearance in the system of a polymer $(p,q)$ when coagulate to any other $(p',q')$ for the benefit of a new polymer with  configuration $(p+p',q+q')$.

The problem \eqref{eq:polymer}-\eqref{eq:constraint} is completed by a Dirichlet boundary condition:
\begin{equation}\label{eq:boundary_polymer}
f = 0, \qquad  \text{on } \ \partial S\,,
\end{equation}
which of course suppose suitable assumptions on the characteristics discussed later. Finally, we require two initial conditions:
\begin{equation} \label{eq:init}
 f(t=0,\cdot)= f^{in}  \  \text{on } \ S, \ \text{and} \ u(t=0)= u^{in}\,. 
\end{equation}
%

\subsection{Contents of the paper and related works}\label{sec:related_works}

The remainder of this paper is devoted to the existence of global solutions to problem \eqref{eq:polymer}-\eqref{eq:constraint}, its numerical approximation, simulations and a discussion on the long-time behaviour.

In Section \ref{sec:analysis}, we establish in Theorem \ref{thm:wellposed} the existence of global solutions. The proof is based on fixed point theorems two treat the non-linear terms, one for the coagulation operator \eqref{eq:Qplus}-\eqref{eq:Qminus}, the second for the constraint \eqref{eq:constraint}. The technique used, implying hypotheses on the regularity of the coefficients, is based on the works on Lifshitz-Slyozov (LS) equation in \cite{Collet2000} and LS with encounters (coagulation) in   \cite{Collet1999}. Techniques also adapted in \cite{Helal2013} for biological polymers. The LS equation is a size structured model for clusters (polymers or more general) formation by addition-depletion of monomers \cite{Lifshitz1961}, while LS with encounters also take into account merging clusters . The coagulation (only size-dependent) is part of the class of coagulation-fragmentation (CF) equation, where fragmentation is the reverse  operator (break-up, splitting of clusters), and it has been studied from a  
mathematical point of view for instance in \cites{Dubovskii1996,Laurencot2000}, also in \cites{Laurencot2002,Amann2005} for the CF equation with space diffusion, and in \cite{Broizat2010} which generalized CF equation with a kinetic approach. The particularity of the model here, is the boundary conditions and the conservation involves. Both necessitates a careful attention in the proof of existence. The former requests to treat the characteristics that come from the boundary, the latter need a particular attention due to the nature of the configurational space.

In Section \ref{sec:scheme}, we propose a finite volume scheme to construct numerical solutions 
approaching the problem. The numerical scheme is in the spirit of the works made in \cite{Bourgade2008} and  \cite{Filbet2008}, where the authors propose a reformulation of the CF equation in a manner well-adapted to a finite volume scheme, namely a conservative form. Here, we write a similar conservative form, but as a cross derivative with respect to both variables. For suitable numerical scheme, we also refer to \cite{Filbet2004} for LS,  to  \cite{Goudon2013a} for LS with encounters and to \cite{Goudon2012} for a model with space diffusion. Then the rest of the section is devoted to the proof of the convergence result given in Theorem \ref{thm:convergence}. It is based on $L^1$ weak compactness of sequences of approximations. We emphasis that the numerical scheme introduced here takes its originality in the introduction of the cross-derivative for the coagulation operator. The main difficulty reside in its approximation which involves new conservation. Moreover, we get a better regularity in the weak 
stability principle result than in \cite{Bourgade2008}, which may also apply to their case.  

 In section \ref{sec:illustration}, we present a numerical simulation which is produced by the numerical scheme. Based on this simulation, we discuss and interpret the long-time behaviour of the model and we show how it might be related to the behaviour of a non-autonomous coagulation equation.
 Here, it might be very interesting to connect this problem to \cites{Herrmann2012,Collet2002,Gabriel2012} on the LS and related model or \cites{Escobedo2005,Desvilettes2009} for the coagulation equation. 

\section{Rescaling and existence of global solutions} \label{sec:analysis}

\subsection{Rescaled problem}
The nature of the configurational space $S$ is not really convenient for computations both for the theoretical point of view and the numerical implementation. Thus, we decide to rescale the problem and operate a change of variable in the distribution $f$ with respect to a physically relevant new variable $r$. We call it the ion ratio (without dimension) and its definition is, for $(p,q)\in S$,
\[ r \coloneqq \frac q p \in(0,1)\,.\]
We introduce then the new unknown $\tilde f$ defined, over the new configuration space $\Sp\coloneqq\R_+\times(0,1)$,  by
\[\tilde f(t,p,r) = p f(t,p,rp)\,.\]
Then, we operate a a change of variable in the association-dissociation and coagulation rates, by introducing $\tilde \V$ define over $\R\times\Sp$ and $\tilde a$ over $\Sp\times\Sp$, such that
\begin{equation*}
 \tilde{\mathcal V}(u,p,r) =  \V(u,p,rp), \text{ and } \ \tilde a(p,r;p',r')=   a(p,rp;p',r'p')\,.
\end{equation*} 
Thus, they satisfy
\begin{equation} \label{eq:V_kl}
 \tilde \V(u,p,r) = \tilde k(p,r) u - \tilde l(p,r)\,,
\end{equation}
with $\tilde k(p,r)=k(p,rp)$ and $\tilde l(p,r)=l(p,rp)$. Also, the symmetry assumption \eqref{eq:a_symmetric} becomes 
\begin{equation} \label{eq:a_symmetric_resc}
 \tilde a(p,r;p',r') = \tilde a(p',r';p,r)\,.
\end{equation}
A formal computation leads to the assessment 
\begin{equation*}
 \partial_t \tilde f(t,p,r) + \frac{1}{p} \partial_r \left(\tilde \V(u(t),p,r) \tilde f(t,p,r) \right) = p Q(f,f)(t,p,q)\,.
\end{equation*}
Finally, letting $\tilde Q =  \tilde Q^+ - \tilde Q^-$ such that
\begin{align}
   \tilde Q^+(\tilde f,\tilde f)(p,r) & =   \frac 1 2  \int_0^p \int_{0}^{1} \frac{p}{p-p'} \tilde a(p',r',p-p',r^*) \indic{(0,1)}(r^*) \nonumber \\ 
  &  \phantom{=   \frac 1 2  \int_0^p \int_{0}^{1} \frac{p}{p-p'} \tilde a(p',r',} \times \, \tilde f(p',r') \tilde f(p-p',r^*) \, d r' d  p' \,, \label{eq:Qplus_resc}\\ 
   \tilde Q^-(\tilde f,\tilde f)(p,r) & = \tilde L(\tilde f)(p,r) \tilde f(p,r) \nonumber \\
 &  \qquad \text{with} \quad  \tilde L(\tilde f)(p,r) =  \int_0^\infty \int_0^{1} \tilde a(p,r;p',r') \tilde f(p',r')\, d r' d p'\,, \label{eq:Qminus_resc}
\end{align} 
with $r^*= \frac{rp-r'p'}{p-p'}$, we get 
\[\tilde Q(\tilde f,\tilde f)(t,p,r) = p  Q(f,f)(t,p,rp)\,.\]
In the following and for the rest we drop tildes in the rescaled problem, for sake of clarity. Now, we are able to reformulate the problem which is to find the distribution $f$ satisfying 
\begin{equation}\label{eq:polymer_resc}
 \frac{\partial f}{\partial t}  + \frac 1 p \frac{\partial}{\partial r}  \left( \V f \right) =  Q(f,f), \qquad \text{on }\  \R_+ \times \Sp\,,
\end{equation}
with the constraint
\begin{equation}\label{eq:constraint_resc}
u(t) + \iint_{\Sp} r p f(t,p,r) dr dp = \rho, \qquad  \text{on } \ \R_+\,.
\end{equation}
and boundary condition \eqref{eq:boundary_polymer} remains given by
\begin{equation}\label{eq:boundary_resc}
f = 0, \qquad \text{on }\ \partial \Sp \,, 
\end{equation}
while the initial conditions is only a change of variables in \eqref{eq:init}:
\begin{equation} \label{eq:init_resc}
 f(t=0,\cdot)= f^{in}  \  \text{on } \Sp, \ \text{and} \ u(t=0)= u^{in}\,. 
\end{equation}

%

%
%
%

%
%
%


\subsection{Hypotheses and result}

The study of the problem \eqref{eq:polymer_resc}-\eqref{eq:constraint_resc} with \eqref{eq:boundary_resc} and \eqref{eq:init_resc} requires hypothesis whether they naturally arise in the problem or technicals. Namely, we assume that:

\medskip

\noindent \textbf{H1.} The initial distribution $f^{in} \in L^1\left(\Sp,(1+p)drdp\right)$ is nonnegative and $u^{in}\geq 0$ such that 

\begin{equation} \label{eq:regularity_fin_uin}
\rho \coloneqq u^{in} + \iint_{\Sp} rp f^{in}(p,r)\, dr dp < +\infty \,. 
\end{equation}

\noindent \textbf{H2.} The coagulation rate $a\in L^\infty(\Sp\times\Sp)$ is nonnegative, satisfies \eqref{eq:a_symmetric_resc} and 

\begin{equation} \label{hyp1}
  \norm{a}_{L^\infty} \leq K\,.
\end{equation}

\noindent \textbf{H3.} The rates functions $p \mapsto  k(p,\cdot),\ l(p,\cdot) \in L^\infty(\R_+;W^{2,\infty}(0,1))$ are both nonnegatives and
\begin{equation} \label{hyp2}
 \norm{k}_{L^\infty(\R_+;W^{2,\infty}(0,1))} + \norm{l}_{L^\infty(\R_+;W^{2,\infty}(0,1))} \leq K \,.
\end{equation}
and for all $p\in \R_+$,
\begin{equation} \label{hyp2bis}
 \norm{k(p,\cdot)}_{W^{2,\infty}(0,1)} + \norm{l(p,\cdot)}_{W^{2,\infty}(0,1)} \leq Kp \,.
\end{equation}

\noindent \textbf{H4.} For all $u\geq 0$ and $p\in\R_+$,  
\begin{equation}\label{hyp3ter}
 \V(u,p,r=0) \geq 0, \text{ and } \, \V(u,p,r=1) \leq 0\,,
\end{equation}
and 
\begin{equation} \label{hyp3bis}
\partial_r \V(u,p,r) = \partial_r k \ u - \partial_r l \leq 0 \quad \text{a.e.} \ (u,p,r)\in \R_+\times\Sp\,.
\end{equation}
Here, $K>0$ denotes a constant. Note that \eqref{hyp3ter} ensures the characteristics remain in the set $\Sp$ and allows us to prescribe the boundary \eqref{eq:boundary_resc}. In fact, it is equivalent with respect to \eqref{eq:V_kl} and (H4) to assume
\begin{equation} \label{hyp3}
 k(p,0) \geq 0, \; l(p,0) = 0 \; \text{and}\; k(p,1) = 0, \; l(p,1) \geq 0\,,
\end{equation}
for all $p \in \R_+$.

\begin{remark}
 With such variables, we note that example \eqref{eq:example_V} becomes
 \begin{equation*} 
 \frac 1 p \V(u(t),p,r) = k_0 p^{\alpha-1} (1-r)^\alpha u(t) - l_0 p^{\beta-1} r^\beta \,.
\end{equation*} 
And, hypothesis (H4) is consistent with this example.
\end{remark}

Now, we are in position to give a definition of the solutions to the problem \eqref{eq:polymer_resc}-\eqref{eq:constraint_resc}.
\begin{definition}[weak solution 1] \label{def:solution}
Let $T>0$ and the initial conditions $f^{in}$ and $u^{in}$ satisfying (H1). A weak solution to \eqref{eq:polymer_resc}-\eqref{eq:constraint_resc} on $[0,T)$ is a couple $(f,u)$ of nonnegative functions such that
\begin{equation} \label{eq:regularity_f}
 f\in C\left([0,T);w-L^1(\Sp)\right)\cap L^\infty\left([0,T),L^1(\Sp,pdrdp)\right)\,,
\end{equation}
and $u \in C([0,T))$, satisfying for all $t\in[0,T)$ and $\varphi\in \mathcal C^1_c (\R_+\times[0,1])$
\begin{multline} \label{eq:distrib_sol}
 \int_S f(t,p,r) \varphi(p,r) \, dr dp - \int_S f^{in}(p,r) \varphi(p,r) \, dr dp \\ 
 = \int_0^t \int_S  \frac 1 p \V(u(s),p,r) f(s,p,r) \partial_r \varphi(p,r) \, dr dp \, ds  \\
 + \int_0^t \int_S Q(f,f)(s,p,r) \varphi(p,r) \, dr dp \, ds\,,
\end{multline}
together with \eqref{eq:constraint_resc} 
\end{definition}

We remark here that regularity \eqref{eq:regularity_f}, where $C\left([0,T);w-X\right)$ means continuous from $[0,T)$ to $X$ a Banach space with respect to the weak topology of $X$. Hypothesis (H1) to (H3) suffice to define \eqref{eq:distrib_sol}. Particularly, \eqref{eq:Qplus_resc}-\eqref{eq:Qminus_resc} entail, as we will see later, that $Q(f,f)$ belongs to $L^\infty\left(0,T;L^1(\Sp)\right)$. 


\medskip

We can now state the main result: 

\begin{theorem}[Global existence] \label{thm:wellposed}
 Let $T>0$. Assume that $f^{in}$ and $u^{in}$ satisfy (H1) and that hypotheses (H2)-(H4) are fulfilled. Then, there exists a  solution $(f,u)$ to the problem \eqref{eq:polymer_resc}-\eqref{eq:constraint_resc} in the sense of Definition \ref{def:solution}. Moreover, the solution has the regularity
 \[ f\in C\left([0,T);L^1(\Sp)\right)\,, \]
 with both
 \[ \int_S f(t,p,r) dr dp \leq \int_S f^{in}(p,r) \, dr dp \,,   \]
 and
 \[ \int_S p f(t,p,r) dr dp = \int_S pf^{in}(p,r) \, dr dp\,.\]
\end{theorem}

Proof of Theorem \ref{thm:wellposed} relies on 2 main steps, which are similar to the ones used for instance in \cite{Collet1999} and \cite{Collet2000} for LS equation. The first step consists in the construction of a mild solution $f$ of equation \eqref{eq:polymer_resc} for a given $u$.  This is achieved through a fixed point theorem by virtue of the contraction property of the coagulation operator. The second step follows a second fixed point which associates \eqref{eq:polymer_resc} to the constraint \eqref{eq:constraint_resc} on $u$. 

Since the method is rather classical, we only provide in the next section the key arguments of the proof, by highlighting the differences between our problem and LS equation with encounter. Particularly, the treatment of the characteristics.

\subsection{Existence of solutions}


\subsubsection{The autonomous problem.}

We start the analysis of the problem for a given nonnegative $u\in C([0,T])$ with $T>0$, \emph{i.e.} we avoid the difficulty induced by the constraint \eqref{eq:constraint_resc}. A well-know approach is to construct the characteristics of the transport operator. These are the curves parametrized by  $p\in \R_+$ and associated to $u$ given, for any $(t,r)\in [0,T]\times (0,1)$, by the solution of

\begin{equation*}
 \begin{array}{l}
  \ds \frac{d}{d s} R_p(s;t,r) = \frac 1 p \V(u(s),p,R_p(s;t,r)) \quad \text{on } [0,T] \\[0.8em]
  \ds R_p(t;t,r) = r\,.
 \end{array}
\end{equation*} 
According to \eqref{hyp2}-\eqref{hyp2bis} and \eqref{hyp3}, there exists a unique solution $R_p(\cdot;t,r)\in C^1([0,T])$. We only consider the characteristic while they are defined, \emph{i.e} $R_p(s;t,r)\in(0,1)$. We would first remark that (H4) ensures the characteristics remain into $(0,1)$ for any $s\geq t$. Moreover, we define the origin time $\sigma_p(t,r) = \inf\{ s \in [0,t] : 0<R_p(s;t,r) <1\}$. So, from the characteristics curves we construct the so-called mild-formulation which is $f$ solution of 
\begin{equation} \label{eq:sol_mild}
 f(t,p,r)  = \begin{cases}
	    \ds f^{in}(p,R_p(0;t,r))J_p(0;t,r) & \\[0.8em]
	    \ds \hspace{1em}+ \int_0^t Q(f,f)(s,p,R_p(s;t,r))J_p(s;t,r) \, ds\,, & \text{if } \sigma_p(t,r)=0 \\[0.8em]
	     \ds \int_{\sigma_p(t,r)}^t Q(f,f)(s,p,R_p(s;t,r))J_p(s;t,r) \, ds\,, & \text{otherwise.}
 \end{cases}
\end{equation}
for all $t\in[0,T]$ and a.e. $(p,r) \in \Sp$ and
 \[J_p(s;t,r) \coloneqq \frac{\partial  R_p}{\partial r}(s;t,r) =  \exp\left( - \int_s^t \frac 1 p (\partial_r \V)(\sigma,p,R_p(\sigma;t,r))\, d\sigma  \right)\,.\]
Note that, for an enough regular solution, the boundary condition \eqref{eq:boundary_resc} is satisfy by \eqref{eq:sol_mild} since $\sigma_p(t,0)=\sigma_p(t,1)=t$. 

Then, to prove Theorem \ref{thm:wellposed} we need to recover the notion of weak solution from the one of  mild solution. This is given by the following the result: 

\begin{lemma}
If $f^{in}\in L^1(\Sp)$ and $Q(f,f)\in L^1((0,T)\times\Sp)$ then, Then the following statements are equivalent:
\begin{enumerate}[i)]
 \item $f\in C([0,T],L^1(\Sp))$ and is solution in the weak sense, \emph{i.e.} satisfies \eqref{eq:distrib_sol}. 
 \item $f$ is a mild solution, \emph{i.e.} satisfies \eqref{eq:sol_mild}.
\end{enumerate}
\end{lemma}
This result is well known and comes from a change of variable and an identification process, we refer to \cites{Collet1999,Collet2000} for LS equation or  \cite{Diperna1989} for Boltzmann equation. The only delicate point remain in the treatment of the origin time and to this end we should proceed as in \cite{Helal2013}. The monotonicity hypothesis \eqref{hyp3bis} is crucial to separate continuously the characteristics coming from $0$ and $1$ and hence to be able to construct the weak solution from \eqref{eq:sol_mild}. Now, with the help of this lemma, it is sufficient to prove the existence of a mild solution. 

Before claiming the existence of mild solution for a given $u$, let us introduce some \emph{a priori} properties of the coagulation operator. Namely, for any $f$ and $g$ both belongs to $L^1\left(\Sp\right)$, we have 

\begin{align}
  &\ds \norm{Q(f,f)}_{L^1(\Sp)} \leq   2 K \norm{f}_{L^1(\Sp)}^2,  \label{eq:L1_Q} \\[0.8em]
  &\ds \norm{Q(f,f) - Q(g,g)}_{L^1(\Sp)} \leq  2 K \left( \norm{f}_{L^1(\Sp)} + \norm{g}_{L^1(\Sp)} \right) \norm{f - g}_{L^1(\Sp)}\,.  \label{eq:contract_Q}
\end{align}
These two estimates ensure that $f \mapsto Q(f,f)$ maps $L^1(\Sp)$ into itself and is Lipschitz  on any bounded subset of $L^1(\Sp)$. Finally, we would remark that for any $f\in L^1(\Sp)$ and $\vphi \in L^\infty(\Sp)$, it holds
\begin{multline} \label{eq:Q_identity}
 \iint_{\Sp} Q(f,f)(p,r) \vphi(p,r)\, dr dp = \frac 1 2 \iint_{\Sp \times \Sp} a(p,r;p',r')f(p,r)f(p',r') \\
 \times \Big[ \vphi(p+p',r^\#) - \vphi(p,r) -\vphi(p',r') \Big] dr'dp'drdp\,.
\end{multline}
with $r^\# = (rp+r'p')/(p+p') \in (0,1)$, which is called weak formulation of the coagulation operator, see \cites{Desvilettes2009,Escobedo2003,Laurencot2002}. We obtain this identity by inversion of integrals applying Fubini's theorem, then changes of variable. In particular,  when $\vphi = \indic{\Sp}$,
\begin{equation} \label{eq:Q_negative}
 \iint_{\Sp} Q(f,f)(p,r) \, dr dp \leq 0\,.
\end{equation}
And, if moreover $f\in L^1(\Sp,pdrdp)$, then 

\begin{equation} \label{eq:pQ_null}
 \iint_{\Sp} p Q(f,f)(p,r) \, dr dp = 0\,.
\end{equation}

Now, for a given $\rho>0$, we define the set:
\[\mathcal B = \{ u\in C([0,T]) : 0\leq u(t) \leq \rho \}\,,\]
and we are ready to claim the next proposition. 

\begin{proposition} \label{prop:autonomous} 
Let $T>0$ and $\rho>0$ with $u$ belongs to the associated set $\mathcal B$. If $f^{in} \in L^1\left(\Sp,(1+p)drdp\right)$, then there exists a unique nonnegative mild solution, i.e. satisfying \eqref{eq:sol_mild}, with
\[f\in L^\infty\left(0,T;L^1\left(\Sp,(1+p) dr dp\right)\right)\,.\]
Moreover, for all $t\in(0,T)$ we have
\begin{equation}  \label{eq:estim_moment0_mild}
 \iint_{\Sp} f(t,p,r) \, dr dp \leq \iint_{\Sp} f^{in}(p,r) \, dr dp\, , 
\end{equation}
and
\begin{equation}  \label{eq:estim_moment1_mild}
 \iint_{\Sp} pf(t,p,r) \, dr dp = \iint_{\Sp} p f^{in}(p,r) \, dr dp\,. 
\end{equation}
\end{proposition} 

\begin{proof}
Here we only give a sketch of the proof. 

\medskip

\noindent \textit{Step 1. Existence and uniqueness.} The local existence of a unique nonnegative solution  $f\in L^\infty\left(0,T';L^1(\Sp)\right)$, for some  $T'>0$ small enough, readily follows from the Banach fixed point theorem applied to the operator that maps $f$ to the right-hand side of \eqref{eq:sol_mild} on a bounded subset of $L^\infty\left(0,T';L^1(\Sp)\right)$. To that, we follow line-to-line \cite{Collet1999} using properties \eqref{eq:L1_Q} and \eqref{eq:contract_Q}.

Then, the global existence, for any time $T>0$, is obtained using estimation \eqref{eq:estim_moment0_mild}, indeed by a classical argument we construct a unique solution on intervals $[0,T']$, $[T',2T']$, \emph{etc}. So, it remains to prove  \eqref{eq:estim_moment0_mild}, which directly follows from the integration of \eqref{eq:sol_mild}, using that $f^{in}\in L^1(\Sp)$ and \eqref{eq:Q_negative}. 

\medskip

\noindent \textit{Step 2. Mass conservation.} It remains to prove \eqref{eq:estim_moment1_mild} which needs  $f\in L^\infty\left(0,T;L^1(\Sp,pdrdp)\right)$ and $Q(f,f)$ too. But, identity \eqref{eq:Q_identity} holds only for $\vphi\in L^\infty(\Sp)$ and \emph{a priori} $pQ(f,f)$ is not integrable. It is enough to follow \cite{Collet1999}*{Lemma 4} which involves a regularization procedure using $a_P(p,r;p',r') = a(p,r;p',r') \indic{(0,P)}(p) \indic{(0,P)}(p')$ in \eqref{eq:sol_mild} and construct a sequence of approximation $f_P\in  L^\infty\left(0,T;L^1(\Sp)\right)$. Then, computing the $L^1$ norm of $f-f_P$ with the mild formulation yields to  the strong convergence of $f_P$ toward the solution $f$ in $L^\infty\left(0,T;L^1(\Sp)\right)$ when $P\rightarrow +\infty$, since $a_P \rightarrow a$ \emph{a.e.} $\Sp\times\Sp$. Finally, because $a_P$  has a compact support, $pf_P$ is integrable and from (H2) we get the uniform bound (in $P$):
\[ \iint_{\Sp} p f_P(t,p,r) \leq C(T)\,,\]
for some constant $C(T)>0$ obtained by a Gronwall's lemma on the mild formulation. We get $f\in L^\infty\left(0,T;L^1(\Sp,(1+p)drdp)\right)$ and $pQ(f,f)$ is integrable. We conclude coming back to \eqref{eq:sol_mild} and integrating it against $p$, which yields \eqref{eq:estim_moment1_mild} thanks to  \eqref{eq:pQ_null}.
\end{proof}

\begin{remark}
 The boundedness of $a$ is a key ingredients in the estimation used in the proof above. Indeed, without it the control of the mass could break down, \emph{i.e.} in finite time $\iint_{\Sp} p f(t,p,r) \,drdp = +\infty$, known as gelation phenomena. Nevertheless, the condition could be relaxed, generally up to a sub-linear coagulation kernel, see for instance \cites{Escobedo2003}.  
\end{remark}

We close this section by stating an additional regularity on the pseudo-moment of the mild solution, key argument to couple the constraint \eqref{eq:constraint_resc} to \eqref{eq:polymer_resc}. 

\begin{corollary}
 Under hypotheses of Proposition \ref{prop:autonomous},  
\[M(t) \coloneqq \iint_{\Sp} r p f(t,p,r) \, dr dp \leq  \iint_{\Sp} p f^{in}(p,r) \, dr dp \,, \] 
and $ M \in W^{1,\infty}([0,T])$ with 
\begin{equation}\label{eq:deriv_moment}
M'(t) = \iint_{\Sp} \V(u(t),p,r)f(t,p,r)\, dr dp\,.
\end{equation}
\end{corollary}

\begin{proof}
The first estimation is a direct consequence of Proposition \ref{prop:autonomous}. Then, we use the formulation \eqref{eq:distrib_sol} and as a test function we let $\varphi_\varepsilon(p,r)= r\xi_\varepsilon(p)$ such that  $\xi_\varepsilon \in C^1_c(\R_+$ and $\xi_\varepsilon(p) = p$ over $ (2\varepsilon, 1/2\varepsilon)$ with $supp \, \xi \subset (\varepsilon, 1/\varepsilon)$, thus
\begin{multline*}
\iint_{\Sp} r f(t,p,r) \xi_{\varepsilon}(p) \, dr dp  = \iint_{\Sp} r f^{in}(p,r) \xi_{\varepsilon}(p) \, dr dp \\ 
 + \int_0^t \iint_{\Sp}  \frac 1 p \V(u(s),p,r) f(s,p,r)  \xi_{\varepsilon}(p) \, dr dp \, ds  \\
 + \int_0^t \iint_{\Sp} r Q(f,f)(s,p,r) \xi_{\varepsilon}(p) \, dr dp \, ds\,,
\end{multline*}
Since we have $f\in L^\infty\left(0,T;L^1(\Sp,(1+p)drdp)\right)$, by the use of \eqref{hyp2bis} and \eqref{eq:pQ_null}, with the Lebesgue theorem we pass to the limit $\varepsilon\rightarrow 0$ and we get 
\begin{equation*}
M(t) = \iint_{\Sp} rp f^{in}(p,r) \, dr dp + \int_0^t \iint_{\Sp}  \V(u(s),p,r) f(s,p,r)  \, dr dp \, ds\,.
\end{equation*}
So, we conclude that
\begin{equation*}
 \frac{d}{dt} M(t)  =  \iint_{\Sp}  \V(u(t),p,r) f(t,p,r)  \, dr dp\,.
\end{equation*}
\end{proof}

We are now ready to apply the second fix point to connect $u$ and $f$ in the system.  

%
%

\subsubsection{Fix point on $u$.}

Again we follow \cites{Collet1999,Collet2000}, \emph{i.e} we let $T>0$ and we define the map 
\[ \mathcal M : u\in \mathcal B \mapsto  \tilde u = \left[ \rho - \iint_{\Sp} r p f_u(t,p,r) \,dr dp\right]_+ \,, \]
where $[\,\cdot\,]_+$ is the positive part and $f_u$ the unique mild solution on $[0,T]$ associated to $u$ thanks to Proposition \ref{prop:autonomous}. It follows that $\mathcal M$ maps $\mathcal B$  into itself. Moreover, using \eqref{eq:deriv_moment}  
\[t\mapsto \rho - \iint_{\Sp} r p f_u(t,p,r) \,dr dp \in W^{1,\infty}(0,T)\,,\]
then,  since $[\,\cdot\,]_+$ is Lipschitz, it holds that $\tilde u \in W^{1,\infty}(0,T)$ with 

\begin{equation*}
 \frac{d}{dt} \tilde u =\begin{cases}
                         0\,, &  \text{if } \ds \iint_{\Sp} r p f_u(t,p,r) \,dr dp \geq \rho\,, \\
                         \ds - \iint_{\Sp} \V(u(t),p,r)f_u(t,p,r)\, dr dp \,, & \text{otherwise.}
                        \end{cases}
\end{equation*}
a.e. $t\in(0,T)$, see \cite{Ziemer1989}*{Theorem 2.1.11}. Thus, for any $u\in \mathcal B$ by using  hypothesis on the rates (H3), it yields
\[ 
 \norm{\frac{d}{dt} \tilde u}_{L^{\infty}(0,T)} \leq  K(\rho+1) \norm{f^{in}}_{L^1(\Sp)}\,.
\]
Next, we invoke Ascoli theorem to claim that   $\mathcal M (\mathcal B) \subset \mathcal B$ is relatively compact in $C([0,T])$. Now, a Schauder fix point theorem would achieve the proof of Theorem \ref{thm:wellposed}. It remains to prove the continuity of the map $\mathcal M$. Let $(u_n)_n$ be a sequence of $\mathcal B$ converging to $u$ for the uniform norm. We need to prove that
\[ \lim_{n\rightarrow +\infty} \norm{\tilde u_n - \tilde u}_{L^\infty(0,T)} = 0\,,\]
which is done by estimating
\begin{multline*}
 \sup_{t\in(0,T)} \abs{\iint_{\Sp} r p f_{u_n}(t,p,r) \,dr dp - \iint_{\Sp} r p f_u(t,p,r) \,dr dp} \\
 \leq \sup_{t\in(0,T)} \iint_{\Sp}  p \abs{f_{u_n}(t,p,r)- f_u(t,p,r)} \,dr dp\,.
\end{multline*}
Indeed, the right hand side of this inequality goes to zero following line-to-line the proof of \cite{Collet1999}*{Lemma 5 and 6} to conclude on the one hand the continuity and on the other that, in fact, 
\[ \iint_{\Sp} r p f_u(t,p,r) < \rho\, \quad \forall t\in[0,T]\,,\]
to drop $[\,\cdot\,]_+$. Thus, there exist $u\in\mathcal B$ such that
\[ u(t) = \mathcal M(u) = \rho - \iint_{\Sp} r p f_u(t,p,r) \geq 0\,.\]
This achieves the proof of Theorem \ref{thm:wellposed}. 

\section{Numerical approximation} \label{sec:scheme}

\subsection{A conservative truncated formulation} \label{ssec:truncate_pb}

The discretization of the problem \eqref{eq:polymer}-\eqref{eq:constraint} gives rise to three main difficulties. First, the unboundedness of the space. Indeed, one of the two variables has been reduced to the interval $(0,1)$, with a physical meaning, but the $p$-variable can reach any size in $\R_+$. Thus, we decided to proceed as in \cite{Bourgade2008} and carry out a truncation of the problem considering a ``maximal reachable size'', or cut-off, $P>0$. The link between both problems, truncated and full, when $P\rightarrow +\infty$ is not taken in consideration here. The reader can refer to step 2 in the proof of Proposition \ref{prop:autonomous} to get some hints related to this topic. The purpose is to provide a converging numerical approximation of a truncated problem for a fixed $P$. The second issue arises when we look toward conservations of the system. In \cite{Bourgade2008}, the authors propose a reformulation of the coagulation operator into a divergence form. We are inspired by this method and 
adapt it to our problem. 
Indeed, this formulation appears natural for finite volume scheme and has the advantage to provide exact conservations at the discrete level. 

The starting point is the weak formulation of $Q$, see \eqref{eq:Q_identity}. One can take $\vphi(u,v) = u\indic{(0,p)}(u) \indic{(0,r)}(v)$ for some $(p,r)\in\Sp$ and  we formally get an expression of the form 
\[\frac{\partial  C}{\partial p \partial r} = p Q(f,f) \,\]
where the coagulation reads now 
\begin{multline} \label{eq:def_C}
  C(f,f)(p,r) \hfill \\ = \int_0^p \int_0^1 \int_0^{p-u} \int_0^1 u a(u,v;u',v') \indic{(0,r)}(v^\#)  f(u,v) f(u',v') \, dv'du'dv du  \\
  \hfill - \int_0^p \int_0^r u L(f)(u,v) f(u,v) \, dv du\,,
\end{multline}
where $v^\# = (uv+u'v')/(u+u')$. Now the coagulation operator has been reformulated in a manner well adapted to a finite volume scheme (the volumes averages of $f$ are brought out directly). It remains  to truncate the problem. It can be achieved in two different ways as mentioned in \cites{Bourgade2008,Filbet2004}, the authors discuss about conservative and non-conservative form. These two options can  be derived respectively by taking $a \coloneqq a \indic{(0,P)}(u+u')$ or $a \coloneqq a \indic{(0,P)}(u)\indic{(0,P)}(u')$. The first option avoids the formation of clusters larger than $P$ thus it will preserve the mass, while the second induces a loss of polymers due to the creation of larger clusters than $P$. This latter is convenient to study gelation phenomenon, see \cite{Escobedo2003} for a review on coagulation. Here, we restrict ourself to the conservative form and obtain the truncated operator by taking
\[ L_P(f)(u,v) = \int_0^{P-u} \int_0^1 a(u',v';u,v) f(u',v') dv' du'\,. \]
Then, replacing $L$ by $L_P$ in \eqref{eq:def_C} it yields, for any $(p,r)\in \Sp_P \coloneqq (0,P)\times(0,1)$, to
\begin{multline} \label{eq:def_CP}
  C_P(f,f)(p,r) \hfill \\ = \int_0^p \int_0^1 \int_0^{p-u} \int_0^1 u a(u,v;u',v') \indic{(0,r)}(v^\#)  f(u,v) f(u',v') \, dv'du'dv du  \\
  \hfill - \int_0^p \int_0^r u L_P(f)(u,v) f(u,v) \, dv du\,. 
\end{multline}
The last main issue lies in the discretization of \eqref{eq:constraint}, \emph{i.e.} the algebraic constraint driving $u$. We will not be able to properly derive an approximation of $\indic{(0,r)}(v^\#)$ in the coagulation operator which would allow us to control the sign of $\rho - \int _{\Sp_P} rpf(t,p,r)\,drdp$. Once again, we reformulate this constraint obtaining an evolution equation on $u$ by a time derivation of it. Thus the problem \eqref{eq:polymer}-\eqref{eq:constraint} reads now
\begin{equation} \label{eq:polymers_trunc}
 p\frac{\partial f}{\partial t}  +  \frac{\partial}{\partial r}  \left( \V f \right) =  \frac{\partial  C_P(f,f)}{\partial p \partial r}, \, \quad (t,p,r)\in \R_+ \times \Sp_P\,,
\end{equation}
and
\begin{equation}\label{eq:ions_trunc}
 \frac{d}{dt}u(t) = - \iint_{\Sp_P} \V(u(t),p,r)f(t,p,r) \ dr dp,\quad t\geq 0\,.
\end{equation}
The boundary condition reads now 
\begin{equation}\label{eq:boundary_trunc}
 f = 0, \qquad \text{on } \ \partial \Sp_P\,,
\end{equation}
and the initial data are
\begin{equation} \label{eq:initial_trunc}
 f(t=0,\cdot) = f^{in} \text{ over } \Sp_P  \text{ and } u(t=0) = u^{in}\,.
\end{equation}
It is now appropriate to introduce the technical assumptions used through this section:

\medskip

\noindent \textbf{H1'.} The initial distribution  $f^{in} \in L^1(\Sp_P)$ is nonnegative and $u^{in}\geq 0$, with

\begin{equation} \label{eq:regularity_fin_uin_trunc}
\rho \coloneqq u^{in} + \iint_{\Sp_P} rp f^{in}(p,r)\, dr dp < +\infty \,. 
\end{equation}

\noindent \textbf{H2'.} The coagulation rate $a\in L^\infty(\Sp_P\times\Sp_P)$ is nonnegative and 

\begin{equation} \label{hyp1_num}
  \norm{a}_{L^\infty} \leq K\,.
\end{equation}

\noindent \textbf{H3'.}  The rate functions $p \mapsto k(p,\cdot),\ l(p,\cdot) \in L^\infty\left(0,P;W^{1,\infty}(0,1)\right)$  are nonnegatives and 
\begin{equation} \label{hyp2_num}
 \norm{k}_{L^\infty\left(0,P;W^{1,\infty}\right)} + \norm{l}_{L^\infty\left(0,P;W^{1,\infty}\right)} \leq K\,.
\end{equation}

\noindent \textbf{H4'.} The function $r \mapsto  \V(u,p,r)$ is a non-increasing function:

\begin{equation} \label{hyp3_num}
\partial_r \V(u,p,r) = \partial_r k \ u - \partial_r l \leq 0 \quad \text{a.e.} \ (u,p,r)\in \R_+\times\Sp_P\,.
\end{equation}
%

\medskip

\begin{remark}
We emphasize that hypotheses H2' and H3' are not so restrictive in front of the truncation, it could allow unbounded rate on the full configuration space $\Sp$ locally bounded which seems reasonable.
\end{remark}

We are now in position to give an alternative definition to our problem \eqref{eq:polymer_resc}-\eqref{eq:constraint_resc}. 

\begin{definition}[Weak solutions 2] \label{def:solution_numeric}
Let $T>0$, a cut-off $P>0$ and let $f^{in}$ and $u^{in}$ satisfying (H1'). 
A weak solution to \eqref{eq:polymers_trunc}-\eqref{eq:ions_trunc} on $[0,T)$ is a couple $(f,u)$ of nonnegative functions, such that
\begin{equation} \label{eq:regularity_def_num}
 f\in C\left([0,T);L^1(\Sp_P)\right) \ \text{ and } \ u \in C([0,T))
\end{equation}
satisfying for all $t\in[0,T)$ and $\varphi\in C^2(\Sp_P)$
\begin{multline}\label{eq:weak_f_num}
  \iint_{\Sp_P}   pf(t,p,r) \varphi(p,r) \,  dr dp \, = \iint_{\Sp_P} p f^{in}(p,r) \varphi(p,r) \, dr dp  \hfill \\
 +  \int_0^t \iint_{\Sp_P} \left( \V(u(s),p,r) f(s,p,r) \frac{\partial \varphi}{\partial r} (p,r) + C_P(s,p,r) \frac{\partial\varphi }{\partial p \partial r} (s,p,r)\right) dr dp ds \\
 \hfill - \int_0^t  \int_0^1  C_P(s,P,r) \frac{\partial \varphi }{\partial r} (s,P,r) \, dr ds - \int_0^t\int_0^{P} C_P(s,r,1) \frac{\partial \varphi}{\partial p} (s,p,1)\,  dp ds \,,
\end{multline}
with
\begin{equation}\label{eq:weak_u_num}
 u(t) = u^{in} - \int_0^t \iint_{\Sp_P} \V(u(s),p,r) f(s,p,r) \, dp dr ds\,.
\end{equation}
\end{definition}

\begin{remark}[Consistence of Definition \ref{def:solution_numeric}]
First, we emphasize that weak formulation \eqref{eq:weak_f_num} is classically obtained after multiplying \eqref{eq:polymers_trunc} by $p$, then integrating over $(0,t)\times\Sp_P$. An integration by parts with respect to the boundary conditions \eqref{eq:boundary_trunc}. Note that the two last integrals in the right hand side correspond to the remaining terms coming from the integration of the coagulation.
Second, the regularity \eqref{eq:regularity_def_num} of $f$ together with the definition of the coagulation operator \eqref{eq:def_CP} provide that for any $t\in[0,T)$ we have $C_P(f,f) \in L^\infty\left((0,t)\times\Sp_P\right)$, $C_P(f,f)(p=P)\in L^\infty\left((0,t)\times(0,1)\right)$ and $C_P(f,f)(r=1) \in L^\infty\left((0,t)\times(0,P)\right)$.
Third, by virtue of hypothesis (H3), for any $U>0$ we get
\begin{equation*} 
\sup_{u\in(0,U)} ||\V(u,\cdot)||_{L^\infty(\Sp_P)} \leq  ||k ||_{L^\infty(\Sp_P)}U + ||l||_{L^\infty(\Sp_P)} \leq K(U+1)\,.
\end{equation*}
Thus, $\V \in L^\infty\left([0,t)\times\Sp_P\right)$. This ensures that equation \eqref{eq:weak_f_num}-\eqref{eq:weak_u_num} are well defined under such regularity and hypotheses. 
\end{remark}

\begin{remark}
A solution in the sense of Definition \ref{def:solution_numeric} regular enough, with an initial datum compactly supported in $(0,P)$, is also a solution in the sense of Definition \ref{def:solution}, \emph{i.e} on the entire space $\Sp$, at least up to a time $T$ small enough. Indeed, since $Q$ is Lipschitz by \eqref{eq:contract_Q}, the speed of propagation of the support of $f$ is finite.
\end{remark}

\begin{remark} 
In general, the definition can be relaxed by taking the solution $f$ belongs to $C\left([0,T);w-L^1(\Sp_P, pdrdp)\right)$ which is sufficient to define the formulation \eqref{eq:weak_f_num}. Nevertheless, we will see that the sequence of approximation is in fact equicontinuous for the strong topology of $L^1(\Sp_P)$ thus Definition \ref{def:solution_numeric} remains stronger but true.
\end{remark}

\subsection{The numerical scheme and convergence statement} \label{ssec:scheme}

This section is devoted to introduce an approximation of the truncated problem presented in Section \ref{ssec:truncate_pb}. Thus in the remainder of this section, both, the truncation parameter $P>0$ and the time parameter $T>0$ are fixed. Our aim is to provide a discretization of $[0,T]\times\Sp_P$ on which we will approach the problem (\ref{eq:polymers_trunc}-\ref{eq:ions_trunc}). Once the scheme is established, we present the main result, namely the convergence in a sense defined later.

Formulation \eqref{eq:polymers_trunc} allow us to use a finite volume method for the configuration space. This is approaching the average of the solution on volume controls at discrete times $t_n$ for $n\in\{0,\ldots,N\}$ such that
\[t_n =n\Dt \quad \text{with} \quad \Dt = T/N \quad \text{and}\quad  N\in\mathbb N^*\,.\]
We turn now to the discretization of the configuration space $\Sp_P$. For sake of simplicity, we consider a uniform mesh of $\Sp_P$ that is given, for some large integer $J$ and $I$, by  $\left(\Lambda_{j,i}\right)_{(j,i)\in \{0,\ldots,J\}\times\{0,\ldots,I\}}$ where
\[ \Lambda_{j,i} = (\pjmh,\pjph)\times(\rimh,\riph) \subset \Sp_P\,,\]
such that $(p_{j-1/2})_{j\in\{0,\ldots,J+1\}}$ and $(r_{i-1/2})_{i\in\{0,\ldots,I+1\}}$ are given by
\[ \pjmh = j \Dp  \quad \text{and} \quad \rimh = i \Dr\,,\]
with $\Dp = P/(J+1)<1$ and $\Dr = 1/(I+1)$.

\begin{remark}
We believe that for a non-uniform mesh it would work, we refer for instance to \cite{Bourgade2008} for a non-uniform discretization of the so-called coagulation-fragmentation equation.
\end{remark}

The average of the solution $f$ to \eqref{eq:polymers_trunc}  at a time $t_{n+1}$ on a cell $\Lambda_{j,i}$ is obtain by integration of \eqref{eq:polymers_trunc} over $[t_{n},t_{n+1})\times\Lambda_{j,i}$ and dividing by the volume of the cell $|\Lambda_{j,i}| = \Dp \Dr$. We aim to derive an induction, in order to obtain an approximation of this average at time $t_{n+1}$, knowing the approximation at time $t_n$ given by  
\[f^n_{j,i} \approx \frac{1}{|\Lambda_{j,i}|} \int_{\Lambda_{j,i}} f(t_{n},p,r)\, dr dp\,,\]
The integration of  \eqref{eq:polymers_trunc} lets appear two types of fluxes which need to be approached. First, the transport term that accounts for the association-dissociation phenomenon given by the $r$-derivative, leads to the numerical fluxes $(F^n_{j,i-1/2})_{(j,i)\in\{0,\ldots,J\}\times\{0,\ldots,I+1\}}$, consistent approximation of:
\begin{equation} \label{eq:flux_consistent}
F^n_{j,i-1/2} \simeq \frac{1}{\Dp} \int_{t_n}^{t_{n+1}}\int_{p_{j-1/2}}^{p_{j+1/2}} \V(u(t),p,r_{i-1/2}) f(t,p,r_{i-1/2}) \, dp dt\,,
\end{equation}

We will use an Euler explicit scheme in time $t$. Moreover, we use the so-called first order upwinding method  to get
\begin{equation} \label{flux_r}
 F^n_{j,i-1/2} = \V^{n +}_{j,i-1/2} f^n_{j,i-1} - \V^{n -}_{j,i-1/2} f^n_{j,i}\,,
\end{equation}
where the velocity at the interface, in function of $u^n\approx u(t_n)$, is given by 
\[\V^n_{j,i-1/2} = \V(u^n,p_{j-1/2},r_{i-1/2})\,,\] 
and using the notation $x^+ = \max(x,0)$ and $x^-=\max(-x,0)$ for any $x\in\R$. The boundary are conventionally taken, for any $j\in \{0,\ldots,J\}$, by 
\begin{equation} \label{eq:flux_r_boundary}
F^n_{j,-1/2} = F^n_{j,I+1/2} =0\,,
\end{equation}
which is in accordance with \eqref{eq:boundary_trunc} and consistent with the approximation \eqref{eq:flux_consistent}. Then the fluxes of coagulation given by the second order derivative is also approached by an Euler explicit method in time, namely our fluxes read

%
%
%
\begin{multline} \label{flux_C}
C^n_{j-1/2,i-1/2} =  \sum_{j'=0}^{j-1} \sum_{i'=0}^I \sum_{j''=0}^{(j-1)-j'} \sum_{i''=0}^I p_{j'} a_{j',i';j'',i''} \delta_{j',i';j'',i''}^{i-1} f^n_{j',i'} f^n_{j'',i''} (\Dp\Dr)^2 \hfill \\
\hfill - \sum_{j'=0}^{j-1} \sum_{i'=0}^{i-1}  \sum_{j''=0}^{J-j'} \sum_{i''=0}^I p_{j'} a_{j',i';j'',i''}  f^{n}_{j',i'} f^n_{j'',i''}   (\Dp\Dr)^2\,,
\end{multline}
where the discrete coagulation rate is 
\begin{equation} \label{eq:approach_a}
  a_{j,i;j,i'} =  \frac{1}{|\Lambda_{j,i}|\times |\Lambda_{j',i'}|} \int_{\Lambda_{j,i}\times\Lambda_{j',i'} } a(p,r;p',r')\, dr dp dr'dp'\,,
\end{equation}
and the characteristic function $\indic{(0,r)}(v^\#)$ is approached by
\begin{equation} \label{approach_indic}
 \delta_{j',i';j'',i''}^{i-1} = \begin{cases}  1 & \text{if } V^\#_{j',i';j'',i''} = \frac{r_{i'+1/2}p_{j'+1/2} + r_{i''+1/2}p_{j''+1/2} }{ p_{j'-1/2}+p_{j''-1/2} } < r_{i-1/2} \\ 0 & \text{otherwise.} \end{cases}
\end{equation}
Finally, we use the convention that
\begin{equation} \label{eq:boundary_C}
 C^n_{-1/2,i-1/2} = C^n_{j-1/2,-1/2} = 0, \quad \forall (j,i)\,,
\end{equation}
which is consistent with the fact that $C_P(t,0,r) = C_P(t,p,0)=0$ by definition of $C_P$ in \ref{eq:def_CP}.

Now, we are almost ready to define the scheme. Indeed, it remains to approach the initial data \eqref{eq:initial_trunc} which are classically obtained, for $(j,i)\in \{0,\ldots,J\}\times\{0,\ldots,I\}$, by 
\begin{equation} \label{eq:initial_f0_trunc}
f_{j,i}^{0} = \frac{1}{|\Lambda_{j,i}|} \int_{\Lambda_{j,i}} f^{in}(p,r)\ dr dp\,.
\end{equation}
and which is nothing but
\begin{equation}\label{eq:initial_u0_trunc}
 u^0 = u^{in}\,.
\end{equation}
Thus, the scheme is defined as follows. 

\begin{definition}[Numerical scheme] \label{def:scheme}
Let us consider the discretization  mentioned above and a given initial data \eqref{eq:initial_f0_trunc}-\eqref{eq:initial_u0_trunc}. The numerical scheme gives us a sequence $(f^n_{j,i})_{n,j,i}$ and $(u^n)_n$, for $n\in\{0,\ldots,N\}$ and $(j,i)\in\{0,\ldots,J\}\times\{0,\ldots,I\}$, defined recursively by 
\begin{equation}\label{eq:scheme_f}
 p_j f^{n+1}_{j,i} =  p_j f^n_{j,i} - \frac{\Dt}{\Dr}\left( F^n_{j,i+1/2}  - F^n_{j,i-1/2} \right)   + \frac{\Dt}{\Dr\Dp}C^n_{j,i}\,,
\end{equation}
and
\begin{equation}\label{eq:scheme_u}
u^{n+1} = u^n - \Dt \sum_{j=0}^J \sum_{i=0}^I F^n_{j,i-1/2} \,\Dr\Dp\,,
\end{equation}
where
\begin{equation} \label{eq:flux_Cji}
C^n_{j,i}=\left(C^n_{j+1/2,i+1/2} - C^n_{j+1/2,i-1/2} \right) - \left( C^n_{j-1/2,i+1/2} - C^n_{j-1/2,i-1/2} \right)\,.
\end{equation}
\end{definition}

In the above definition, the coagulation is written with fluxes defined by \eqref{flux_C}. But, note that it can be also expressed as follows
\begin{multline} \label{def_Cji}
C_{j,i}^n =  \sum_{j'=0}^j \sum_{i'=0}^I \sum_{i''=0}^I  p_{j'} a_{j',i';j-j',i''} f^n_{j',i'} f^n_{j-j',i''} \delta_{j',i';j-j',i''}^{i,i-1} \, (\Dp\Dr)^2  \\
-p_{j} \left(  \sum_{j'=0}^{J-j} \sum_{i'=0}^I   a_{j,i;j',i'}  f^n_{j',i'} \Dp\Dr \right) f^{n}_{j,i} \, \Dp\Dr\,.
\end{multline}
where 
\begin{equation*}
 \delta_{j',i';j-j',i''}^{i,i-1} = \begin{cases}  1 & \text{if } r_{i-1/2} \leq V^\#_{j',i';j-j',i''}  < r_{i-1/2} \\ 0 & \text{otherwise.} \end{cases}
\end{equation*}
This is obtained when reordering summation behind \eqref{eq:flux_Cji}. Such a formulation is not only simpler to implement numerically, but also useful in several estimations in the next section. We also remark, by virtue of  \eqref{eq:boundary_C} and \eqref{eq:flux_Cji}, that the coagulation satisfy  
\begin{equation} \label{eq:C_conserv}
 \sum_{j=0}^J \sum_{i=0}^I C_{j,i}^n =0, \quad \forall n\in \{0,\ldots,N\}\,,
\end{equation}
which will ensure  mass conservation at the discrete level.

Now the scheme is stated, we focus on its convergence. This will be achieved by a well-suited construction of sequences of approximations. For this purpose, we let  $h=\max(\Dr,\Dp,\Dt)$.

\begin{definition}[Sequences of approximations] \label{def:seq_approx}
Let the sequences $(f^n_{j,i})_{n,j,i}$ and $(u^n)_n$ construct by virtue of Definition \ref{def:scheme}. We define the piecewise constant approximation $f_h$  on $[0,T)\times\Sp_P$ by
\begin{equation} \label{eq:piecewise_fh}
 f_h(t,p,r) = \sum_{n=0}^{N-1} \sum_{j=0}^J \sum_{i=0}^I f_{j,i}^n \indic{\Lambda_{j,i}}(p,r) \indic{[t_n,t_{n+1})}(t)\,, 
\end{equation}
and then the piecewise linear (in time) approximation $\tilde f_h$
\begin{multline}\label{eq:piecewise_linear_fh}
  \tilde f_h(t,p,r) = \sum_{n=0}^{N-1}\sum_{j=0}^J \sum_{i=0}^I \left(\frac{f^{n+1}_{j,i}-f^n_{j,i}}{\Dt}(t-t_n) + f^n_{j,i}\right) \\
  \times \indic{\Lambda_{j,i}}(p,r)\indic{[t_n,t_{n+1})}(t)\,.
\end{multline}
Moreover, we define the piecewise approximation of $u$ on $[0,T)$
\begin{equation} \label{eq:piecewise_u}
 u_h(t) = \sum_{n=0}^{N-1} u^n  \indic{[t_n,t_{n+1})}(t), \quad \text{on}\ [0,T)\,.
\end{equation}
\end{definition}
Here, we mention that under such definition both approximations satisfy at time $t=0$ the same initial condition which is given by   $f_h(0,p,r) = \tilde f_h(0,p,r) = f^{in}_h(p,r)$ where
\begin{equation} \label{eq:piecewise_fin}
  f^{in}_h(p,r) =  \sum_{j=0}^J \sum_{i=0}^I f^0_{j,i} \indic{\Lambda_{j,i}}(p,r),  \quad \text{on}\ \Sp_P\,.
\end{equation}
and that $u_h(0) = u^{in}$. We are now ready to state the convergence result we obtain.

\begin{theorem}[Convergence] \label{thm:convergence}
 Let $T>0$, assume that $f^{in}$ and $u^{in}$ satisfy (H1'), and that hypotheses (H2') to (H4)' are fulfilled. Moreover, we make the stability assumption that
 \begin{equation} \label{eq:CFL}
  4\frac{\Dt}{\Dr} || \V ||_{L^\infty((0,U_T)\times\Sp_P)} < 1 \text{ and } 2 K M^{in}(1+P) \Dt < 1\,,
 \end{equation}
 where 
 \begin{equation} \label{eq:def_bound}
  U_T=e^{KM^{in}T}u^{in} \text{ and }  M^{in} =  \int_{\Sp_p} f^{in}(p,r)\, dr dp\,.
 \end{equation}
 Then there exists a couple $(f,u)$ solution of the problem (\ref{eq:polymers_trunc}-\ref{eq:ions_trunc}) in the sense of Definition \ref{def:solution_numeric} such that, up to a subsequence (not relabeled)
 \begin{align}
  & f_h \xrightharpoonup[h\rightarrow 0]{} f, & \ w-L^1((0,T)\times\Sp_P)\,, \label{eq:cv1} \\
  & \tilde f_h \xrightarrow[h\rightarrow 0]{} f, & C\left([0,T];w-L^1(\Sp_P)\right)\,, \label{eq:cv2}\\
  & u_h \xrightarrow[h\rightarrow 0]{} u,  &  \forall t\in [0,T]\,. \label{eq:cv3}
 \end{align}
 with $\norm{f_h - \tilde f_h}_{L^\infty(0,T;L^1(\Sp_P))} \rightarrow 0$ when $h\rightarrow 0$.
\end{theorem}

\begin{remark}
At this stage, we emphasize that hypothesis \eqref{eq:CFL} involved in Theorem \ref{thm:convergence} is classical. The first one is the so-called \emph{Courant-Friedrich-Lax} condition (or CFL condition) which ensures a well-know convex formulation of the transport part. The second is to control the sign of the coagulation term. 
\end{remark}

For technical reasons we consider a carefully reconstruction of the coagulation term. This one involves the coagulation kernel which is approached by
\begin{equation*}
 a_h(p,r;p',r') = \sum_{j=0}^J \sum_{i=0}^{I} \sum_{j'=0}^J \sum_{i'=0}^{I} a_{j,i;j',i'} \indic{\Lambda_{j,i}}(p,r) \indic{\Lambda_{j',i'}}(p',r')\,.
\end{equation*}
where the $a_{j,i;j',i'}$ are given by \eqref{eq:approach_a}. A classical result of piecewise approximation, since $a$ satisfy \eqref{hyp1_num}, is  
\begin{equation*}
 a_h \underset{h\rightarrow 0}{\longrightarrow} a, \quad L^1(\Sp_P).
\end{equation*}
Further, we mention that  $a_h$ converges towards $a$ a.e.  $\Sp_P\times\Sp_P$, which holds true up to a subsequence. Also, we will reconstruct the characteristic function \eqref{approach_indic}, this delicate point responsible of a lack of conservation in our scheme will be discussed later, as well as the rate $\V$.
 
\begin{remark}
 In the remainder of the paper most of the result will involve more extraction of sequences, so that, for sake of clarity, we use the same index even for all the sequences, even if we extract a new subsequence of the previous one.  
\end{remark}

\subsection{Estimations and weak compactness}


The aim of this section is to introduce the well-suited estimations leading to the required compactness to prove Theorem \ref{thm:convergence}. The method used here has been extensively developed in the field of coagulation-fragmentation equations and/or LS equation and their derivatives, see Section \ref{sec:related_works} for more details. Briefly, here we proceed in two steps. We provide some discrete properties of the scheme. Then, we establish compactness results on both $f_h$ and $\tilde f_h$ and we treat $u_h$ by classical results on sequences of bounded variation functions.

First of all, we introduce here some useful notation which will help us lighten the next computation:
\begin{align*}
 & P_h(p) = \sum_{j=0}^J p_j \indic{[\pjmh,\pjph)}\,,& &R_h(r) = \sum_{i=0}^I r_i \indic{[\rimh,\riph)}\,,  \\
 & P_h^\pm(p) = \sum_{j=0}^J p_{j\pm 1/2} \indic{[\pjmh,\pjph)}\,,& &R_h^\pm(r) = \sum_{i=0}^I r_{i\pm 1/2} \indic{[\rimh,\riph)}\,,
\end{align*}
and
\[ \Theta_{\Dt}(t) = \sum_{n=0}^{N-1} t_n \indic{[t_n,t_{n+1})}(t)\,.\] 
Moreover, for sake of conciseness we denote $\sigma \coloneqq \{0,\ldots,J\}\times\{0,\ldots,I\}$, when no confusion on $I$ and $J$ holds. Then, as long as it does not entail any ambiguity, we use
\[\begin{array}{lcll}
   \sigma^* &= &\{0,\ldots,J-j\}\times\{0,\ldots,I\} & \text{if indexed on } (j',i')\,, \\
   \sigma^{**} &=& \{0,\ldots,J-j'\}\times\{0,\ldots,I\}  & \text{if indexed on } (j'',i'')\,. 
  \end{array}
\]
Also, we denote $l=(j,i)$, $l'=(j',i')$ and $l''=(j'',i'')$. Thus, the summation notation in a compact form are
\begin{equation*}
 \sum_{j=0}^J \sum_{i=0}^I = \sum_{l\in\sigma}, \, \sum_{j=0}^J \sum_{i=0}^I \sum_{j'=0}^{J-j} \sum_{i'=0}^I = \sum_{l\in\sigma}\sum_{l'\in\sigma^*}, \, \sum_{j'=0}^J \sum_{i'=0}^I \sum_{j''=0}^{J-j'} \sum_{i''=0}^I = \sum_{l'\in\sigma}\sum_{l''\in\sigma^{**}}.
\end{equation*}

\subsubsection{Discrete estimations}

To begin, we establish some properties of the sequences constructed in Definition \ref{def:scheme}. We emphasize that the original continuous problem involves, a decrease of the moment of order 0 \eqref{eq:Q_negative}, while the $p$-moment of order 1 \eqref{eq:pQ_null} is conserved and the total balance \eqref{eq:constraint} remains constant. This latter will be discussed later, while the two other remain true at the discrete level and will be part of the next proposition.  To that, we introduce the discrete moments of the sequences defined in Definition \ref{def:scheme}, given by
\begin{equation} \label{eq:moment_discrete}
M^n_{0,h} =  \sum_{l\in\sigma} f_{j,i}^n \,\Dr \Dp, \quad \text{and} \quad M^n_{1,h} =  \sum_{l\in\sigma} p_j f_{j,i}^n \,\Dr \Dp\,.
\end{equation}
So, the next proposition establish the basic properties of our scheme and particularly of these moments.

\begin{proposition}[Non-negativeness, moments and conservation] \label{prop:discrete}
 Let $f^{in}$ and  $u^{in}$ satisfying (H1'), together with $(f^n_{j,i})_{(n,j,i)\in\sigma}$ and $(u^n)_n$ construct by virtue of Definition \ref{def:scheme}. We assume that  the stability condition \eqref{eq:CFL} holds true. Then, the sequences $(f^{n}_{j,i})_{(n,j,i)}$ and $(u^{n})_n$  are both  nonnegatives and satisfy for all $n\in\{0,\ldots,N-1\}$:
 \[ 0 \leq M_{0,h}^{n+1} \leq M_{0,h}^{n} \leq M^{in} \,,\]
 and $n\in\{0,\ldots,N\}$
 \[  0\leq M_{1,h}^{n} = M_{1,h}^0  \quad \text{and} \quad 0\leq  u^{n} \leq U_T\,,\]
 where $M^{in}$ and $U_T$ are both given in \eqref{eq:def_bound}. Moreover, there exist a constant $C>0$ independent on $h$ such that for all $n\in\{0,\ldots,N-1\}$ we have
  \[\sum_{(j,i)\in\sigma} |f^{n+1}_{j,i} -f^n_{j,i}| \Dr\Dp \leq C \Dt \quad \text{and} \quad  |u^{n+1}-u^n| \leq C \Dt\,. \]
\end{proposition}


\begin{proof}
We prove this proposition by induction.  We suppose that $(f^n_{j,i})_{(j,i)}$ is a nonnegative sequence and $u^n$ a nonnegative data, both given for some $n\in\{0,\ldots,N-1\}$ satisfying 
\[ 0 \leq M_{0,h}^n \leq M^{in}, \quad \text{and}  \quad  u^n \leq (1+  \Dt K M^{in})^n u^0 \,.\]
We easily check that it is true for $n=0$. Indeed, the non negativeness is given by hypothesis (H1) on $f^{in}$ and $u^{in}$ together with the initial approximation \eqref{eq:initial_f0_trunc}-\eqref{eq:initial_u0_trunc}. Then, by the definitions of $f^{in}_h$ in \eqref{eq:piecewise_fin} and the constant $M^{in}$ in \eqref{eq:def_bound}, we get
\begin{align} 
& M^{0}_{0,h} = \iint_{\Sp_P} f^{in}_h(p,r)\ dr dp = M^{in}, \label{eq:bound_moment0}
\intertext{and}
& M^{0}_{1,h} =  \iint_{\Sp_P} P_h(p) f^{in}_h(p,r)\ dr dp \leq P M^{in}. \label{eq:bound_moment1}
\end{align}
The rest of the proof is separated in four steps. We start by estimate the moments, next we bound $u^{n+1}$, then we prove the non negativeness of $f^{n+1}$ and finally we prove the last two ``time'' estimations of the proposition.

\smallskip

\noindent \textit{Step 1. Moments estimation}. We first remark the fluxes $F^n$ are null at the boundary, it is \eqref{eq:flux_r_boundary}, thus
\[ \sum_{i=0}^I (F^{n}_{j,i+1/2} - F^{n}_{j,i-1/2})  = 0, \quad \forall j \,.\]
It remains to estimate the contribution of the coagulation in the moments. The first order moment is naturally conserved from the construction of our scheme, it is \eqref{eq:C_conserv}. These two remarks lead, by  equation \eqref{eq:scheme_f} giving the $f^{n+1}_{j,i}$, to the fact that 
\[M_{1,h}^{n+1} = M_{1,h}^{n}\,.\]
Now, for the zeroth order moment. We remark that $p_{j'}/p_{j+j'} -1 \leq 0$ and by non negativeness of $(f^n_{j,i})_{(j,i)\in\sigma}$ and $a_{j,i,j',i'}$ together with its symmetry, we get 
\begin{multline*}
 \sum_{l \in\sigma } \frac{1}{p_j} C^n_{j,i} =   \sum_{l \in\sigma}  \sum_{l' \in\sigma^*}   \frac{p_{j}}{p_{j+j'}} a_{j,i;j',i'} f^{n}_{j,i} f^{n}_{j',i'}\,(\Dp\Dr)^2  \hfill \\
\hfill -   \sum_{l \in\sigma} \sum_{l' \in \sigma^*}   a_{j,i;j',i'}  f^{n}_{j',i'}   f^{n}_{j,i}\,(\Dp \Dr)^2 \leq 0\,,
\end{multline*}
where the equality is obtained from expression \eqref{def_Cji} after inverting and re-indexing the summation.
%
%
%
It proves the discrete zeroth order moment decreases which is the desired property.

\smallskip

\noindent \textit{Step 2. Properties of $u^{n+1}$}. By definition of the fluxes $F^n$ in \eqref{flux_r}, the non negativeness of $(f^n_{j,i})_{(j,i)}$ and hypothesis (H3') on the rate $\V$, we get that 
\[F^n_{j,i-1/2} \leq  \V^{n+}_{j,i-1/2} f_{j,i-1}^n \leq  K u^n f_{j,i-1}^n \,.\]
which implies by the expression of $u^{n+1}$ in \eqref{eq:scheme_u} and since we assumed that $M^n_{0,h} \geq 0$ is bounded by $M^{in}$, we get
\begin{equation*}
u^{n+1}  \geq \left( 1  - K \Dt  \sum_{l \in \sigma}  f_{j,i}^n \Dp \Dr\right) u^n \geq \left( 1  - K M^{in} \Dt \right) u^n\, .
\end{equation*}
%
%
%
This latter entails the non  negativeness of $u^{n+1}$ by the stability assumption \eqref{eq:CFL}. It remains to bound $u^{n+1}$. Indeed, we have
\[ u^{n+1} \leq u^n + \Dt \sum_{l \in \sigma} V^{n-}_{j,i-1/2} f_{j,i} \Dp \Dr \leq  (1 + \Dt K  M^{in}) u^n\,.\]

\smallskip

\noindent\textit{Step 3. Non negativeness of $f^{n+1}$}. We prove this result by studying separately the transport part and the coagulation part. On the one hand, using the definition of the fluxes $F^n$ in \eqref{flux_r} and since we have $\V^n_{j,i-1/2} = \V^{n+}_{j,i-1/2} - \V^{n-}_{j,i-1/2}$, we get the following decomposition
\begin{multline} \label{flux_reformulation}
   F^n_{j,i+1/2} - F^n_{j,i-1/2}   =  (\V^{n}_{j,i+1/2} - \V^{n}_{j,i-1/2})  f^n_{j,i} \hfill \\
\hfill   + \V^{n -}_{j,i+1/2} (f^n_{j,i}-f^n_{j,i+1})+ \V^{n +}_{j,i-1/2} ( f^n_{j,i} - f^n_{j,i-1})\,. 
\end{multline}
Now, we denote by
\begin{equation} \label{eq:def_aji}
  A_{j,i}^n = \frac{1}{\Dr} \int_{\rimh}^{\riph} \left(-\frac{\partial}{\partial r}\V(u^n,\pjmh,r) \right) dr  = \frac{\V^{n}_{j,i-1/2} - \V^{n}_{j,i+1/2}}{\Dr}\,,
\end{equation}
which is nonnegative for any $(j,i)\in\sigma$ by the monotonicity hypothesis (H4'). It results from the formulation \eqref{flux_reformulation} that 
\begin{multline} \label{eq:convexity}
 \frac 1 2 f^n_{j,i} - \frac{\Dt}{\Dr} \left( F^n_{j,i+1/2} - F^n_{j,i-1/2} \right) = \Dt A_{j,i}^n f^n_{j,i} \hfill \\
 + \frac 1 4 \Big[ \big(1- 4\frac{\Dt}{\Dr}\V^{n -}_{j,i+1/2}\big)f^n_{j,i} + 4\frac{\Dt}{\Dr}\V^{n -}_{j,i+1/2} f^n_{j,i+1} \Big] \\
 \hfill + \frac 1 4 \Big[ \big(1- 4\frac{\Dt}{\Dr}\V^{n +}_{j,i-1/2}\big)f^n_{j,i} + 4\frac{\Dt}{\Dr}\V^{n +}_{j,i-1/2} f^n_{j,i-1} \Big], 
\end{multline}
Thus, the non negativeness of the $A_{j,i}^n$ in \eqref{eq:def_aji} and the convex combination of the nonnegative $f_{j,i}^n$,  we get 
\begin{equation} \label{eq:pos_inter}
\frac 1 2 f^n_{j,i} - \frac{\Dt}{\Dr} \left( F^n_{j,i+1/2} - F^n_{j,i-1/2} \right)\geq 0\,.
\end{equation}
Thanks to the definition of $f_{j,i}^{n+1}$ in \eqref{eq:scheme_f}, and using \eqref{eq:pos_inter} we obtain

\begin{equation*}
  f^{n+1}_{j,i} \geq  \frac 1 2 f^n_{j,i} + \frac{\Dt}{\Dr\Dp} \frac{1}{p_j} C^n_{j,i}\,.
\end{equation*}
Then, in the definition of $C_{j,i}^n$ we keep only the negative term to get the following estimate
\begin{equation*}
 f^{n+1}_{j,i} \geq \left(  \frac 1 2 -  \Dt \sum_{j'=0}^{J-j} \sum_{i'=0}^I  a_{j,i;j',i'}  f^n_{j',i'} \Dp\Dr \right) f^n_{j,i} \geq \frac 1 2 \left(  1 -  2 K \Dt M^{in} \right) f^n_{j,i} \geq 0\,.
\end{equation*}
where in the two last inequalities we used hypothesis (H2') on the coagulation kernel, the stability assumption \eqref{eq:CFL}, and   \eqref{eq:bound_moment0}.
%

\smallskip

\noindent\textit{Step 4. Time estimation}. From the definition of $f^{n+1}_{j,i}$ in \eqref{eq:scheme_f}, we easily obtain
\begin{equation*}
 |f^{n+1}_{j,i} -f^n_{j,i}|\Dr\Dp  \leq  \Dt  \frac{1}{p_j}|F^n_{j,i+1/2}-F^n_{j,i-1/2}|\Dp + \Dt \frac{1}{p_j} |C_{j,i}^n|\,.
\end{equation*}
On the one hand, we have from the definition of the discrete coagulation in \eqref{def_Cji} and hypothesis (H1')  that
\begin{equation*}
 \sum_{(j,i)\in\sigma} \frac{1}{p_j} |C_{j,i}^n| \leq 2 K M^{in}{}^2\,.
\end{equation*}
On the other hand, since for all $j\in\{0,\ldots,J\}$ we have $p_j \geq \Dp/2$, then  
\begin{equation*}
 \sum_{(j,i)\in\sigma} \frac{\Dp}{p_j}|F^n_{j,i+1/2}-F^n_{j,i-1/2}| \leq 8 M^{in} \sup_{u\in(0,U_T)} ||\V(u,\cdot)||_{L^\infty(\Sp_P)}\,.
\end{equation*}
Thus, summing the first inequality together with the two last ones we prove the required estimation. Finally, we remark that from the expression of $u^{n+1}$ in \eqref{eq:scheme_u}, we have 
\[u^{n+1}-u^n = \Dt  \sum_{(j,i)\in\sigma} F^n_{j,\imh}.\]
We have to bound the fluxes \eqref{flux_r}. Indeed, we have for any $(j,i)\in\sigma$ and $i\neq0$
\[ \abs{F^n_{j,\imh}} \leq \sup_{u\in(0,U_T)} ||\V(u,\cdot)||_{L^\infty(\Sp_P)} ( \fhim + \fh )\,, \]
and then summing over $\sigma$ we get
\[|u^{n+1}-u^n| \leq   2 \sup_{u\in(0,U_T)} ||\V(u,\cdot)||_{L^\infty(\Sp_P)}  M^{in}\Dt\,.\]
It ends the proof.
\end{proof}

This proposition establishes the main properties of our scheme, we now derive a corollary that transposes these properties to the sequences of approximations. 

\begin{corollary} \label{cor:estim_continuous}
 Under hypothesis of Proposition \ref{prop:discrete}. Let the sequences of approximations $(f_h)_h$, $(\tilde f_h)_h$ and $(u_h)_h$ construct by Definition \ref{def:seq_approx}. Then, for all discretization parameter $h$, 
 \[ f_h \in L^\infty(0,T;L^1(\Sp_P)) \quad \text{and} \quad \tilde f_h \in C([0,T];L^1(\Sp_P))  \]
 together with $ u_h \in L^\infty(0,T)$. Moreover, we have for the sequence $(f_h)_h$ the uniform estimation
 \begin{equation} \label{eq:bound_moment0_continuous}
 \iint_{\Sp_P} f_h(t,r,p) \ dr dp \leq  \iint_{\Sp_P} f_h(s,r,p) \ dr dp \leq M^{in}, \quad \forall  t \geq s\,,
 \end{equation}
 and 
\begin{equation}  \label{eq:discrete_moment1_conservation}
\iint_{\Sp_P} P_h(p) f_h(t,r,p) \ dr dp = \iint_{\Sp_P} P_h(p) f_h^{in}(r,p) \ dr dp, \quad \forall t\in[0,T)\,.
\end{equation}
 For the sequence $(\tilde f_h)_h$, we have 
 \begin{equation} \label{eq:discrete_moment0_continuous_tildefh}
  0 \leq \iint_{\Sp_P} \tilde f_h(t,r,p) \ dr dp  \leq M^{in}, \quad \forall t\in[0,T] \,,
 \end{equation}
 and there exist a constant $C>0$ independent on $h$ such that
 \begin{equation} \label{eq:continuity_tildefh}
  \norm{\tilde f_h(t,\cdot) - \tilde f_h(s,\cdot)}_{L^1(\Sp_P)} \leq C |t-s|, \quad \forall s,t\in[0,T]\,.
 \end{equation}
 Finally, the sequence $(u_h)_h$ satisfy the uniform bound
 \begin{equation} \label{eq:born_uniform_u}
   \norm{u_h}_{L^\infty(0,T)} \leq U_T,  \quad  \text{and} \quad \norm{u_h}_{BV(0,T)} < C T\,.
 \end{equation}

\end{corollary}

\begin{proof}
 The regularity and non negativeness of the sequences $(f_h)_h$ and $(u_h)_h$ follow from Proposition \ref{prop:discrete} and Definition \ref{def:seq_approx}. And for now, by the same arguments,  it is clear that $\tilde f_h \in L^\infty(0,T;L^1(\Sp_P))$. Next, \eqref{eq:bound_moment0_continuous} and \eqref{eq:discrete_moment1_conservation} follow from the properties of the discrete moment in Proposition \eqref{prop:discrete}, by a simple reformulation using the definition of $f_h$ in \eqref{eq:piecewise_fh}. Then, for all $n\in\{0,\ldots,N-1\}$ and $t\in[t_n,t_{n+1})$, 
 \begin{equation} \label{eq:convex_tildef}
  \frac{f^{n+1}_{j,i}-f^n_{j,i}}{\Dt}(t-t_n) + f^n_{j,i} = \frac{(t-t_n)}{\Dt}f^{n+1}_{j,i}  + \left(1 - \frac{(t-t_n)}{\Dt}\right) f^{n}_{j,i}  \,.
 \end{equation}
 Thus, by definition of $\tilde f_h$ in \ref{eq:piecewise_linear_fh}, the discrete moment \eqref{eq:moment_discrete} and the convex combination mentioned above, we have for all $t\in[0,T]$:
 \begin{equation*}
 0 \leq \iint_{\Sp_P} \tilde f_h(t,r,p) \ dr dp  \leq \sup_{n\in\{0,\ldots,N-1\}} \max(M^{n+1}_{0,h},M^n_{0,h}) \leq M^{in}\,.
 \end{equation*}
 and that $\tilde f_h$ are nonnegatives.  It provides the uniform bound \eqref{eq:discrete_moment0_continuous_tildefh}. Next, estimation \eqref{eq:continuity_tildefh} is against a direct consequence of Proposition \ref{prop:discrete} (last estimation). It provides the regularity in time of the sequence $\tilde f_h$. Finally, \eqref{eq:born_uniform_u} is a consequence of Proposition \ref{prop:discrete} and the definition of $u_h$ in \eqref{eq:piecewise_u}.
 
\end{proof}

\begin{remark}
Before coming back to the proof of Theorem \ref{thm:convergence}, we emphasize on the fact that, as mentioned before, our scheme preserves the mass \eqref{eq:discrete_moment1_conservation}. Nevertheless, the algebraic condition \eqref{eq:constraint}, transformed into \eqref{eq:ions_trunc} is no more preserved at the discrete level. The following corollary is not used in the demonstration of the convergence, but only states that the lack of mass that occurs in our scheme can be control. Indeed, the deviation from the initial value is of magnitude $\Dr$. 
\end{remark}
\begin{corollary}
 Under hypotheses of Proposition \ref{prop:discrete}, for any $n\in\{0,\ldots,N\}$ we define 
 \[\rho^n \coloneqq u^n + \sum_{(j,i)\in\sigma} r_i p_j f_{j,i}^n,\]
 then we have for some constant $C=C(M^{in},K)\geq0$ that
 \[ |\rho^n - \rho^0| \leq C T \Dr.\]
\end{corollary}

\begin{proof}
We remark that summing \eqref{eq:scheme_f} tested against $r_i$ and with \eqref{eq:scheme_u}, we get
\[ \rho^{n+1} - \rho^{n} = \Dt \sum_{(j,i)\in\sigma} r_i C_{j,i}^n.\]
Then by \eqref{eq:flux_Cji}, we obtain
\[ \sum_{(j,i)\in\sigma} r_i C_{j,i}^n = \sum_{i=0}^I C_{J+1/2,i+1/2} \Dr.\] 
Combining these two equalities and by definition of the flux \eqref{flux_C} and the bound $M^{in}$ in \eqref{eq:def_bound}, we get 
\[ \abs{\rho^{n+1} - \rho^{n}} \leq 2 K P M^{in}{}^2 \Dr \Dt,\]
which ends the proof.
\end{proof}
%
%

\subsubsection{Weak compactness}

We introduced, from the scheme established in Definition \ref{def:scheme}, sequences of approximations in Definition \ref{def:seq_approx} satisfying the properties stated in Corollary \ref{cor:estim_continuous} that suppose to approach the solution to our problem. An important issue in proving this result is to get the convergence, towards some functions, of these sequences in a sense that allow us to obtain an enough regular solution to \eqref{eq:polymers_trunc}-\eqref{eq:ions_trunc}. The answer to this issue is obtain by argument of compactness. In this section, we provide the necessary compactness estimates to pass to the limit.  


\medskip

The first estimate will follow from a refined version of the De La Vall\'ee-Poussin Theorem \cite{Chau-Hoan1976}, also we refer to \cite{Dellacherie1988}*{Chap. II, Theorem 22} for a probabilistic approach. Indeed, since  $f^{in} \in L^1(\Sp_P)$ is nonnegative, there exists  
\begin{equation} \label{eq:def_phi}
 \begin{gathered}
 \phi \in C^1([0,+\infty)) \text{ nonnegative, convex, with concave derivative,}\\
 \text{with }\phi(0) =0, \quad \phi'(0)=0,\quad \text{and} \quad \frac{\phi(r)}{r} \underset{r\rightarrow +\infty}{\longrightarrow} + \infty\,,
 \end{gathered}
\end{equation}
such that 
\begin{equation}\label{eq:estim_phi_fin}
 \iint_{\Sp_P} \phi(f^{in}(p,r))\, dr dp < +\infty\,.
\end{equation}
Therefore, proving that \eqref{eq:estim_phi_fin} can be propagated in time, uniformly according to $h$, will give us the uniform integrability of the sequences $f_h$ and $\tilde f_h$.

%
\begin{lemma} \label{lem:uniform_int}
 Let $\phi$ satisfying \eqref{eq:def_phi} such that \eqref{eq:estim_phi_fin} holds true. Then, there exists $C\geq 0$ independent on $h$, such that 
 %
 %
 for any $t\in [0,T)$, we have 
 \begin{equation}\label{eq:phi_continuous}
  \iint_{\Sp_P} \phi(f_h^{\Dt}(t,p,r)) \, dr dp \leq e^{C T} \int_{\Sp_R} \phi(f^{in}(p,r)) \, drdp\,.
 \end{equation}
 and 
  \begin{equation}\label{eq:phi_continuous_tilde}
  \iint_{\Sp_P} \phi(\tilde f_h^{\Dt}(t,p,r)) \, dr dp \leq e^{C T} \int_{\Sp_R} \phi(f^{in}(p,r)) \, drdp\,.
 \end{equation}
 
\end{lemma}

\begin{proof}
Let us derive first \eqref{eq:phi_continuous_tilde} from \eqref{eq:phi_continuous}. By the definition of $\tilde f_h$ in \eqref{eq:piecewise_linear_fh}, the convex combination \eqref{eq:convex_tildef}, and since $\phi$ is convex by \eqref{eq:def_phi}, we have for any $t\in[0,T]$  
%
%
\begin{multline*}
  \iint_{\Sp_P}\phi(\tilde f_h(t,p,r) ) \ dr dp \leq   \sup_{n\in\{0,\ldots,N\}} \sum_{(j,i)\in\sigma}  \phi(f^n_{j,i}) \Dr\Dp \hfill \\
 \hfill \leq \sup_{t\in[0,T]} \int_{\mathcal S_P} \phi( f_h(t,p,r))\ dr dp
\end{multline*}
Thus, it remains to prove \eqref{eq:phi_continuous} to conclude. We split the computation into two parts, in order to treat first the transport part and then the coagulation. 

\smallskip

\noindent\textit{Step 1. The transport}. This first part involves the convexity of $\phi$ and is closely related to the estimation done in \cite{Filbet2003}*{Lemma 3.5}. Indeed, let us denote the intermediate value: 
\begin{equation} \label{eq:def_tildef}
 \tilde f^n_{j,i} = f^n_{j,i} - \frac{\Dt}{\Dr} ( F^n_{j,i+1/2} - F^n_{j,i-1/2} )\,.
\end{equation}
From the convex formulation \eqref{eq:convexity}, we easily obtain for any $(j,i)\in\sigma$ the following expression
\begin{multline} \label{eq:f_intermediate}
 \ds \tilde  f^n_{j,i} = ( \frac 1 2 + \Dt A_{j,i}^n)  f^n_{j,i} + \frac 1 4 \Big[ \big(1- 4\frac{\Dt}{\Dr}\V^{n -}_{j,i+1/2}\big)f^n_{j,i} + 4\frac{\Dt}{\Dr}\V^{n -}_{j,i+1/2} f^n_{j,i+1} \big] \hfill \\
 \hfill   + \frac 1 4 \Big[ \big(1- 4\frac{\Dt}{\Dr}\V^{n +}_{j,i-1/2}\big)f^n_{j,i} + 4\frac{\Dt}{\Dr}\V^{n +}_{j,i-1/2} f^n_{j,i-1} \Big]
\end{multline}
with the convention $f_{j,I+1} = f_{j,-1} =0$ and $A^n_{j,i}$ the discrete gradient defined in \eqref{eq:def_aji}. Our aim is to write $\tilde  f^n_{j,i}$ as a complete convex combination of the $f^n_{j,i}$'s together with $0$. Thus, let us introduce the coefficients

\[\lambda_{j,i}^0 = \frac 1 2 + \Dt A_{j,i}^n,\]
\begin{equation*}
\begin{aligned}
& \lambda_{j,i}^1 = \frac 1 4 - \frac{\Dt}{\Dr}\V^{n -}_{j,i+1/2}\,, &  \lambda_{j,i}^2 = \frac{\Dt}{\Dr}\V^{n -}_{j,i+1/2}\,, \\
& \lambda_{j,i}^3 = \frac 1 4 - \frac{\Dt}{\Dr}\V^{n +}_{j,i-1/2}\,, &  \lambda_{j,i}^4 = \frac{\Dt}{\Dr}\V^{n +}_{j,i-1/2}\,,
\end{aligned}
\end{equation*}
which are non negatives since $A_{j,i}^n \geq 0$, by monotonicity hypothesis (H4), and by stability assumption \eqref{eq:CFL}. Then, by virtue of hypothesis (H3) and since $u_h$ satisfies the bound \eqref{eq:born_uniform_u}, we get that 
\begin{equation*} 
||\partial_r \V(u,p,r)||_{L^\infty((0,U_T)\times\Sp_P)} =  ||\partial_r k||_{L^\infty(\Sp_P)} U_T + ||\partial_r k||_{L^\infty(\Sp_P)} \leq K_T\,,
\end{equation*}
where $K_T =K(U_T+1)$. Thus, we renormalized the coefficients as follows 
\[ \tilde \lambda_{j,i}^k = \frac{\lambda_{j,i}^k}{1 + 2 K_T \Dt }, \qquad k=0,\ldots,4\,.\]
From \eqref{eq:f_intermediate}, it follows that
\[\frac{\tilde  f^n_{j,i}}{1 + 2 K_T \Dt } = \tilde \lambda_{j,i}^0 f^n_{j,i} + \tilde \lambda_{j,i}^1  f^n_{j,i} + \tilde \lambda_{j,i}^2  f^n_{j,i+1} + \tilde \lambda_{j,i}^3  f^n_{j,i} + \tilde \lambda_{j,i}^4  f^n_{j,i-1}\,.\]
It remains to remark that
\[ 0 \leq 1- \tilde \lambda_{j,i}^5 =  \sum_{k=0}^4 \tilde \lambda_{j,i}^k  = \frac{1 + A_{j,i}^n \Dt}{1+ 2 K_T \Dt} \leq \frac 1 2\,,\]
and we obtain by convexity of $\phi$ and $\phi(0) = 0$ 
\[\phi\left(\frac{\tilde  f^n_{j,i}}{1 + 2 K\Dt }\right) \leq \tilde \lambda_{j,i}^0 \phi(f^n_{j,i}) + \tilde \lambda_{j,i}^1 \phi(f^n_{j,i}) + \tilde \lambda_{j,i}^2  \phi(f^n_{j,i+1}) + \tilde \lambda_{j,i}^3  \phi(f^n_{j,i}) + \tilde \lambda_{j,i}^4  \phi(f^n_{j,i-1})\,.\]
Then, summing over $i$, reordering the sum and remarking that 
\begin{equation*} 
 \tilde \lambda_{j,i}^0  +  \tilde \lambda_{j,i}^1  +  \tilde \lambda_{j,i-1}^2  +  \tilde \lambda_{j,i}^3  +   \tilde \lambda_{j,i+1}^4  = \frac{1}{1+2K_T\Dt}\,,
\end{equation*}
it is straight forward that 
\begin{equation} \label{eq:sum_intermediate}
\sum_{(j,i)\in\sigma} \phi\left(\frac{\tilde  f^n_{j,i}}{1 + 2 K \Dt }\right) \leq\sum_{(j,i)\in\sigma} \phi(f^n_{j,i})\,.
\end{equation}
Finally, using that the derivative of $\phi$ is concave, we have $\phi'(\delta y)\leq \delta \phi'(y)$ for $(\delta,y)\in [1,+\infty)\times\R_+$ and thus integrating over $(0,x)$ this latter, we get that
\begin{equation} \label{eq:concav_phi}
 \phi(\delta x) \leq \delta^2 \phi(x), \quad \forall (\delta,x)\in [1,+\infty)\times\R_+\,.
\end{equation}
We conclude this intermediate estimation, using \eqref{eq:sum_intermediate} and \eqref{eq:concav_phi}, 
\begin{equation}\label{eq:estim_intermediate}
\sum_{(j,i)\in\sigma} \phi(\tilde f^n_{j,i}) \leq  (1 + 2 K_T\Dt)^2 \sum_{(j,i)\in\sigma} \phi( f^n_{j,i})\,. 
\end{equation}

\smallskip

\noindent \textit{Step 2. The coagulation}. Now we get the first part of our estimation, it remains to take into account the coagulation. We estimate the following quantity
\[ \sum_{(j,i)\in\sigma} \phi(f^{n+1}_{j,i}) - \phi(\tilde f^n_{j,i}) \leq \sum_{(j,i)\in\sigma} (f^{n+1}_{j,i} - \tilde f^n_{j,i}) \phi'(f^{n+1}_{j,i})\,,\]
which comes from the convexity of $\phi$. By  definition of  the $\tilde f^n_{j,i}$ in \eqref{eq:def_tildef} together with the expression of $f^{n+1}_{j,i}$ in \eqref{eq:scheme_f}
\begin{equation} \label{eq:convex_coag}
 \sum_{(j,i)\in\sigma} \phi(f^{n+1}_{j,i}) - \phi(\tilde f^n_{j,i}) \leq \frac{\Dt}{\Dp\Dr} \sum_{(j,i)\in\sigma} C_{j,i} \phi'(f^{n+1}_{j,i})\,.
\end{equation}
Non negativity of $f^n$ yields
\[ C_{j,i} \leq   K \sum_{j'=0}^j \sum_{i'=0}^I \sum_{i''=0}^I  p_{j'} f^n_{j',i'} f^n_{j-j',i''} \delta_{j',i';j-j',i''}^{i,i-1} (\Dp\Dr)^2\,.\] 
then summing over $j$ and $i$, we get
\begin{multline} \label{eq:estim_phi_Cji}
  \sum_{(j,i)\in\sigma} \sum_{j'=0}^j \sum_{i'=0}^I \sum_{i''=0}^I  p_{j'} f^n_{j',i'} f^n_{j-j',i''} \delta_{j',i';j-j',i''}^{i,i-1} \phi'(f_{j,i}^{n+1})\, (\Dp\Dr)^2 
 \hfill \\
 \hfill  =  \sum_{(j',i')\in\sigma}   p_{j'} f^n_{j',i'} \left( \sum_{(j'',i'')\in\sigma^*}  f^n_{j'',i''} \phi'(f_{j'+j'',i^\#}^{n+1})\right)  \,(\Dp\Dr)^2
\end{multline}
where $i^\# \in \{0,\ldots,I\}$ such that $\delta_{j',i';j'',i''}^{i^\#,i^\#-1} = 1$. Now, we remark as in \cite{Bourgade2008}*{Lemma 3.2} and proving for instance with the help of \cite{Laurencot2002}*{Lemma B.1} that when $\phi$ fulfills \eqref{eq:def_phi}, we get that
\[ x\phi'(y) \leq \phi(x) + \phi(y), \quad \forall (x,y)\in\R_+^2\,.\]
Using this property and the bound on the first moment \eqref{eq:bound_moment1} in \eqref{eq:estim_phi_Cji}, it follows
\begin{multline} \label{eq:estim_intermediate2}
  \sum_{(j,i)\in\sigma}  C_{j,i} \phi'(f_{j,i}^{n+1})  \hfill \\
   \leq    K \sum_{(j',i')\in\sigma}   p_{j'} f^n_{j',i'} \left( \sum_{(j'',i'')\in\sigma}\phi(f^n_{j'',i''}) + \sum_{(j'',i'')\in\sigma} \phi(f_{j'+j'',i^\#}^{n+1})\right)  (\Dp\Dr)^2 \\
   \hfill \leq K P M^{in} \left( \sum_{(j,i)\in\sigma} \phi(f^n_{j,i}) + \sum_{(j,i)\in\sigma} \phi(f^{n+1}_{j,i}) \right) \Dp\Dr\,.
\end{multline}
Combining both \eqref{eq:estim_intermediate} and \eqref{eq:estim_intermediate2} with \eqref{eq:convex_coag}, we get
\begin{multline*}
 \sum_{(j,i)\in\sigma} \phi(f_{j,i}^{n+1}) \leq   (1 + 2 K_T \Dt )^2  \sum_{(j,i)\in\sigma} \phi(f_{j,i}^{n}) \hfill \\
 \hfill +    P K M^{in} \Dt \left( \sum_{(j,i)\in\sigma} \phi(f^n_{j,i}) + \sum_{(j,i)\in\sigma} \phi(f^{n+1}_{j,i}) \right)\, .
\end{multline*}
or in other term, when $\Dt <1$
\begin{multline*}
(1-PKM^{in}\Dt) \left(\sum_{(j,i)\in\sigma}  \phi(f_{j,i}^{n+1}) - \sum_{(j,i)\in\sigma}  \phi(f_{j,i}^{n}) \right) \hfill \\
\hfill \leq   \left( 4K_T(K_T+1) + 2PKM^{in} \right)\Dt \sum_{(j,i)\in\sigma} \phi(f^n_{j,i})
\end{multline*}
Dividing by $1-PKM^{in} \Dt \geq 1/2$ regarding the stability condition \eqref{eq:CFL}, thus it holds that for any $n\in\{0,\ldots,N\}$
\begin{equation*}
 \sum_{(j,i)\in\sigma}  \phi(f_{j,i}^{n}) \leq  e^{C T}\sum_{(j,i)\in\sigma} \phi(f^0_{j,i})\,,
\end{equation*}
where $C = 8K_T(K_T+1) + 4PKM^{in}$. The conclusion follows from the definition of $f_h$ in \eqref{eq:piecewise_fh} and $f_h^{in}$ in \eqref{eq:piecewise_fin} with the Jensen inequality, since 
\[ \phi(f^0_{j,i}) = \phi\left(\frac{1}{|\Lambda_{j,i}|} \int_{\Lambda_{j,i}}  f^{in}(p,r) \ dp dr \right) \leq  \frac{1}{|\Lambda_{j,i}|} \int_{\Lambda_{j,i}}  \phi\left(f^{in}(p,r)\right) \, dp dr\,.\]
It ends the proof.
\end{proof}

The direct consequence of Lemma \ref{lem:uniform_int} is that $(f_h)_{h}$ is weakly relatively compact in $L^1((0,T)\times\Sp_P)$ as a consequence of the Dunford-Pettis theorem, see \cite{Edwards1995}*{Theorem 4.21.2}. It proves there exists a subsequence (not relabeled) and $f\in L^1((0,T)\times\Sp_P)$ such that  
\begin{equation*}
 f_h \underset{h\rightarrow 0}{\rightharpoonup} f \qquad w-L^1((0,T)\times\Sp_P).
\end{equation*}
At this stage, the convergence is too weak to be able to pass to the limit, particularly in the quadratic term, and to get the final regularity of $f$ in Definition \ref{def:solution_numeric}. To this end, we will use the  piecewise linear in time approximation. Against invoking the Dunford-Pettis theorem, for all $t\in[0,T]$ we have that $\tilde f_h(t,\cdot)$ belongs to a relatively compact subset of $L^1(\Sp_P)$. Then, by Corollary \ref{cor:estim_continuous} we have that the sequence is equicontinuous in time for the strong topology of $L^1(\Sp_P)$, thus for the weak topology. So, applying Ascoli Theorem, there exists a subsequence of $\tilde f^{\Dt}_h$ (not relabeled) converging towards a $g$ in $C\left([0,T];w-L^1(\Sp_P)\right)$. Next, we remark that 
\begin{equation} \label{eq:cv_f-tildef}
 \sup_{t\in(0,T)} || \tilde f_h(t,\cdot) - f_h(t,\cdot) ||_{L^1(\Sp_P)} \leq C \Dt,
\end{equation}
which ensures that $g=f$.  Finally, by weak convergence we get 
\[|| f(t,\cdot) - f(s,\cdot) ||_{L^1(\Sp_P)} \leq \liminf_{h\rightarrow 0} || \tilde f_h(t,\cdot) - \tilde f_h(s,\cdot) ||_{L^1(\Sp_P)} \leq C |t-s|.\]
And this latter prove the continuity for the strong topology of $L^1(\Sp_P)$ of the limit $f$. This achieves the proof of the convergence \eqref{eq:cv1}-\eqref{eq:cv2} towards $f$ (not yet the solution).

%
%
%

But, it remains to prove the convergence \eqref{eq:cv3} of $u_h$ before passing to the limit. Indeed, in Corollary \ref{cor:estim_continuous}  we have \eqref{eq:born_uniform_u} the uniform bound, w.r.t. $h$, in $L^\infty(0,T)\cap BV(0,T)$, then the Helly Theorems, see \cite{Kolmogorov1975}*{Theorem 36.4 and 36.5}, entail that up to a subsequence (not-relabeled) there exist $u\in BV(0,T)$ such that the sequence $(u_h(t))_{h}$ converges towards $u(t)$ for every $t\in [0,T]$. This prove \eqref{eq:cv3}.

%
%

\subsection{Convergence of the numerical scheme}

Here we prove that the limit $f$ and $u$ obtained right before are solutions of the problem \eqref{eq:polymers_trunc}-\eqref{eq:ions_trunc} to conclude the proof of Theorem \ref{thm:convergence}.
 
\subsubsection{Reconstruction and convergence of the coagulation operator} \label{sec:reconstruct_coagulation}

One of the delicate point in the proof of convergence is to give an appropriate reconstruction of the quadratic operator, the coagulation, that convergences in a relevant sense.  In order to perform it, we define over $[0,T)\times\Sp_P$ the following approximation:
\begin{multline*}
 C_{P,h}(t,p,r)   \hfill \\
  = \int_{\Sp_P\times\Sp_P} \Phi^{1,h}_{p,r}(p',r';p'',r'') f_h(t,p',r') f_h(t,p'',r'') \, dr'' dp'' dr' dp' \\
 \hfill - \int_{\Sp_P\times\Sp_P}  \Phi^{2,h}_{p,r}(p',r';p'',r'') f_h(t,p',r')f_h(t,p'',r'') \,  dr'' dp'' dr' dp'.
\end{multline*}
where for any $(t,p,r)\in[0,T)\times\Sp_P$ and $(p',r';p'',r'')\in\Sp_P\times\Sp_P$, 
\begin{multline*}
 \Phi^{1,h}_{p,r}(p',r';p'',r'') = \indic{(0,P_h^-(p))}(p')\indic{(0,P_h^-(p)-P_h^-(p'))}(p'')  \hfill \\
 \hfill \times \ \indic{(0,R_h^+(r))}(V_h^\#(p',r';p'',r'')) P_h(p')a_h(p',r';p'',r''),
\end{multline*}
and
\begin{multline*}
\Phi^{2,h}_{p,r}(p',r';p'',r'') =  \indic{(0,P_h^-(p))}(p')  \indic{(0,P-P_h^-(p'))}(p'')  \hfill \\
\hfill \times \ \indic{(0,R_h^-(r))}(r') P_h(p') a_h(p',r';p'',r''),
\end{multline*}
with
\[V_h^\#(p',r';p'',r'') = \frac{R^+_h(r') P^+_h(p') + R^+_h(r'') P^+_h(p'') }{ P^-_h(p') + P^-_h(p') }.\]
With such definition, for all $n\in\{0,\ldots,N-1\}$ and $(j,i)\in\sigma$, it is straightforward that for any $(t,p,r)\in[t_n,t_{n+1})\times\Lambda_{j,i}$ we have
\[C_h(t,p,r) = C^n_{j-1/2,i-1/2}\,.\]
Now the convergence of $C_h$ will be a consequence of the following to lemma. The first one can be find as is in \cite{Bourgade2008}*{Lemma 3.5}. 
\begin{lemma}\label{lem:cvg}
 Let $\Omega$ be an open set of $\R^m$, $\kappa>0$ and let two sequences $(v_n)_{n\in\mathbb N}\subset L^1(\Omega)$ and $(w_n)_{n\in\mathbb N}\subset L^\infty(\Omega)$. If we assume that for all $n\in \mathbb N$, $|w_n|\leq\kappa$ and there exist $v\in L^1(\Omega)$ and $w\in L^\infty(\Omega)$ satisfying 
 \[ v_n \underset{n\rightarrow+\infty}{\longrightarrow} v, \quad weak-L^1(\Omega), \text{ and } \ w_n \underset{n\rightarrow+\infty}{\longrightarrow} w, \quad \text{a.e. in } \Omega. \]
 Then, 
 \[ \norm{v_n(w_n-w)}_{L^1(\Omega)}\underset{n\rightarrow+\infty}{\longrightarrow} 0, \text{ and } \ v_nw_n \underset{n\rightarrow+\infty}{\longrightarrow} vw,  \quad weak-L^1(\Omega). \]
\end{lemma}
The second lemma give us some useful properties on the functions $\Phi^{i,h}_{p,r}$.
\begin{lemma} \label{lem:cv_phi}
For $i=1,2$, 
\[ ||\Phi^{i,h}_{p,r}||_{L^\infty(\Sp_P\times\Sp_P)} \leq P K, \quad \forall (p,r)\in \Sp_P,\] 
and for all $(p,r)\in \Sp_P$,
\begin{align*}
 & \Phi^{1,h}_{p,r}(p',r';p'',r'') \underset{h\rightarrow 0}{\longrightarrow} \indic{(0,p)}(p')\indic{(0,p-p')}(p'') \indic{(0,r)}(v^\#) p' a(p',r';p'',r''), \\
 & \Phi^{2,h}_{p,r}(p',r';p'',r'') \underset{h\rightarrow 0}{\longrightarrow} \indic{(0,p)}(p')  \indic{(0,P-p')}(p'')  \indic{(0,r)}(r')  p' a(p',r';p'',r'').
\end{align*}
\emph{a.e.} $\Sp_P\times\Sp_P$.
\end{lemma}
\begin{proof}
The first inequality follows from  hypothesis (H2'). Then, we only have to check that $ \indic{(0,R_h^+(r))}(V_h^\#(p',r';p'',r''))$ converge almost every where. Indeed for all $r\in(0,1)$,

\begin{multline} \label{decompose_indic}
\int_{\Sp_P\times \Sp_P} \abs{\indic{(0,R_h^+(r))}(V_h^\#) - \indic{(0,r)}(v^\#)} dr''dp''dr'dp'  \hfill \\
\leq \int_{\Sp_P\times \Sp_P} \abs{\indic{(0,R_h^+(r))}(V_h^\#) - \indic{(0,R_h^+(r))}(v^\#)} dr''dp''dr'dp' \\
\hfill + \int_{\Sp_P\times \Sp_P} \abs{\indic{(0,R_h^+(r))}(v^\#) - \indic{(0,r)}(v^\#)} dr''dp''dr'dp' \,,
\end{multline}
where
\[V_h^\#(p',r';p'',r'') = \frac{R^+_h(r') P^+_h(p') + R^+_h(r'') P^+_h(p'') }{ P^-_h(p') + P^-_h(p') } \leq v^\# = \frac{r'p'+r''p''}{p'+p''},\]
Therefore, the first integral in the right hand side of \eqref{decompose_indic} is reduced to the measure of the set 
\[ A_h = \left\{ (p',r',p'',r'') \in \Sp_P\times\Sp_P\  : \ V_h^\# \leq R_h^+(r) \leq v^\# \right\}.\]
Remarking that $V_h^\#$ converge everywhere to $v^\#$ and $R_h^+(r)$ towards identity, $A_h$ converges towards $v^\#{}^{-1}(r)$ which is a null set for the Lebesgue measure. It remains to remark that the second integral in \eqref{decompose_indic} converges to zero too, and conclude that $\indic{(0,R_h^+(r))}(V_h^\#)$ converges toward $\indic{(0,r)}(v^\#)$ in $L^1(\Sp_P\times\Sp_P)$. This provide us the convergence almost everywhere on $\Sp_P\times\Sp_P$, up to a subsequence (against not relabeled).  
\end{proof}


The sequence $f_h$ do not have a sufficient regularity so, to pass to the limit, the trick is to consider the operator 
\begin{multline*}
 \tilde C_{P,h}(t,p,r) =  \int_{\Sp_P\times\Sp_P} \Phi^{1,h}_{p,r}(p',r';p'',r'') \tilde f_h(t,p',r') \tilde f_h(t,p'',r'') \ dr'' dp'' dr' dp'\\
 - \int_{\Sp_P\times\Sp_P}  \Phi^{2,h}_{p,r}(p',r';p'',r'') \tilde f_h(t,p',r') \tilde f_h(t,p'',r'') \  dr'' dp'' dr' dp'.
\end{multline*}
Here we proceed as in \cite{Bourgade2008}*{Section 4}, applying twice Lemma \ref{lem:cvg} thanks to Lemma \ref{lem:cv_phi}, we get 
\begin{equation*}
 \tilde C_{P,h} \underset{h\rightarrow 0}{\longrightarrow} C_P, \quad \text{on } \ [0,T)\times \Sp_P\,. 
\end{equation*}
Finally, by Lemma \ref{lem:cv_phi}, Corollary \ref{cor:estim_continuous} and the convergence obtained in \eqref{eq:cv_f-tildef}, we have
\begin{multline*}
    |C_{P,h}(t,p,r) - \tilde C_{P,h}(t,p,r) |  \hfill \\ 
    \leq 2 K P \left( \norm{f_h}_{L^\infty(0,T;L^1)} + \lVert \tilde f_h \rVert_{L^\infty(0,T;L^1)} \right) \lVert f_h - \tilde f_h \rVert_{L^\infty(0,T;L^1)} \\
    \hfill \underset{h\rightarrow 0}{\longrightarrow}  0, \quad \forall (t,p,r)\in[0,T)\times \Sp_P\,.
\end{multline*}
Thus, since $C_{P,h} = C_{P,h} - \tilde C_{P,h} + \tilde C_{P,h}$, we have
\begin{equation*}
  C_{P,h}  \underset{h\rightarrow 0}{\longrightarrow} C_P, \quad \text{on } \ [0,T)\times \Sp_P\,. 
\end{equation*}
Moreover, it is obvious that $C_{P,h}$ is bounded by the bound \eqref{eq:bound_moment0_continuous} and Lemma \ref{lem:cv_phi}, then the Lebesgue dominated convergence theorem yields 
\begin{equation*}
  C_{P,h}(t,\cdot) \underset{h\rightarrow 0}{\longrightarrow} C_P(t,\cdot), \quad L^1(\Sp_P) \quad \forall t\in[0,T). 
\end{equation*}

\subsubsection{Final stage of the proof}

The final stage of the proof is to write the  discrete weak formulation of the scheme, when the equation \eqref{eq:scheme_f} is multiplied by discrete test functions $\vphi_{j,i}$, and then to prove that it converges to the continuous weak formulation. Thus, let $\vphi\in C^2(\Sp_P)$ and multiply equation \eqref{eq:scheme_f} by  $\vphi_{j,i}=\vphi(\pjmh,\rimh)$. Then summing over $(j,i)$ and $k=0,\ldots,n-1$ for some $n\in\{1,\ldots,N\}$, we get
\begin{multline*}
  \sum_{k=0}^{n-1} \sum_{j=0}^J \sum_{i=0}^I  p_j f^{k+1}_{j,i} \vphi_{j,i} \Dr \Dp -  \sum_{n=0}^{k-1} \sum_{j=0}^J \sum_{i=0}^I p_j f^k_{j,i}\vphi_{j,i} \Dr \Dp   \hfill \\ \hfill  = - \Dt  \sum_{k=0}^{n-1} \sum_{j=0}^J \sum_{i=0}^I  \left( F^k_{j,i+1/2}  - F^k_{j,i-1/2} \right) \vphi_{j,i}\Dp   +  \Dt  \sum_{k=0}^{n-1} \sum_{j=0}^J \sum_{i=0}^I C_{j,i}^k \vphi_{j,i}.
\end{multline*}
Reordering the sum, and making use of the boundary conditions \eqref{eq:flux_r_boundary} and \eqref{eq:boundary_C}, we infer the following equation
\begin{equation} \label{eq:weak_discrete}
 X^n_h =  Y_h^n + Z_h^n,
\end{equation}
where
\begin{align*}
  X^n_h & = \sum_{j=0}^J \sum_{i=0}^I  p_j f^{n}_{j,i} \vphi_{j,i} \Dr \Dp - \sum_{j=0}^J \sum_{i=0}^I p_j f^{0}_{j,i}\vphi_{j,i} \Dr \Dp, \\
  Y_h^n & = \Dt \sum_{k=0}^{n-1} \sum_{j=0}^J \sum_{i=1}^I  F^k_{j,i-1/2}( \vphi_{j,i} -  \vphi_{j,i-1}) \Dp,\\
  Z_h^n & = \Dt \sum_{k=0}^{n-1}\sum_{j=1}^J \sum_{i=1}^I {C}^k_{j-1/2,i-1/2}\Bigg[(\vphi_{j-1,i-1} - \vphi_{j-1,i}) - (\vphi_{j,i-1} - \vphi_{j,i})\Bigg] \nonumber \\
	& \qquad + \Dt \sum_{k=0}^{n-1} \sum_{j=1}^J {C}^k_{j-1/2,I+1/2} (\vphi_{j-1,I}-\vphi_{j,I} ) \nonumber \\
	& \qquad + \Dt \sum_{k=0}^{n-1}\sum_{i=1}^I {C}^n_{J+1/2,i-1/2}(\vphi_{J,i-1} - \vphi_{J,i}).
\end{align*}
Next, we define $X_h$ on $[0,T)$ by 
\begin{multline} \label{eq:X}
X_h(t) \coloneqq  \iint_{\Sp_P} P_h(p) f_h(t,p,r) \vphi(P_h^-(p),R_h^-(r))\, dr dp \\
- \iint_{\Sp_P} P_h(p) f^{in}_h(p,r) \vphi(P_h^-(p),R_h^-(r))\, dr dp.
\end{multline}
Then, we define $Y_h$ by
\begin{equation}\label{eq:Y}
Y_h(t)= Y_h^1(t) + Y_h^2(t),
\end{equation}
with
\begin{multline*}
 Y_h^1(t)  = \int_{0}^{t} \iint_{\Sp_P}  \indic{\Theta_{h}(t)}(s) \indic{(0,1-\Dr)}(r) \V^{+}(u_h(s),P_h^-(p),R_h^-(r)) f_h(s,p,r)\\ 
 \times D_h^0[\vphi](p,r) \, dr dp ds,
\end{multline*}
and
\begin{multline*}
 Y_h^2(t)  = - \int_{0}^{t} \iint_{\Sp_P} \indic{\Theta_{h}(t)}(s) \indic{(\Dr,1)}(r) \V^{-}(u_h(s),P_h^-(p),R_h^-(r)) f_h(s,p,r)\\
 \times D_h^0[\vphi](p,r) \, dr dp ds,
\end{multline*}
where a Taylor expansion of $\vphi$ gives
\begin{equation*}
 D_h^0[\vphi](p,r) = \frac{\partial \vphi}{\partial r} ((P_h^-(p),R_h^-(p))) + o(\Dr).
\end{equation*}
In the same manner, we define $Z_h$ by
\begin{equation} \label{eq:Z}
Z_h(t) =  Z_h^1(t) + Z_h^2(t) + Z_h^3(t),
\end{equation}
such that
\begin{align*}
  Z_h^1(t) & = \int_0^t \int_{\Sp_R} \indic{(0,\Theta_{h}(t))}(s) {C}_h(s,p,r) D_h^1[\varphi](p,r) \, dr dp  ds \\
  Z_h^2(t) & = \int_0^t \int_0^P \indic{(0,\Theta_{h}(t))}(s) {C}_h(s,p,1) D_h^2[\varphi](p,1) \, dp ds \\
  Z_h^3(t) & = \int_0^t \int_0^1 \indic{(0,\Theta_{h}(t))}(s) {C}_h(s,P,r) D_h^3[\varphi](P,r) \, dr ds 
\end{align*}
with the expansion
\begin{align*}
  D_h^1[\varphi](p,1) & =  \frac{\partial^2 \vphi}{\partial p \partial r}(P_h(p),R_h(r)) + o(\Dp) + o(\Dr),\\
  D_h^2[\varphi](P,r) & =  \frac{\partial \vphi}{\partial r}(P,R_h^-(r)) + o(\Dr),\\
  D_h^3[\varphi](p,r) & =  \frac{\partial \vphi}{\partial p}(P_h^-(p),1) + o(\Dp).
\end{align*}
%
%
It is straightforward that for any $n\in\{1,\ldots,N\}$ and $t\in[t_n,t_{n+1})$ we have 
\[ X_h(t) = X^n_h, \quad  Y_h(t) = Y^n_h, \quad \text{and} \quad Z_h(t) = Z^n_h\,.\]
Thus, by  \eqref{eq:weak_discrete},  it holds that for all $t\in[0,T)$
\begin{equation} \label{eq:reconstruct_continuous_f}
 X_h(t) = Y_h(t) + Z_h(t)\,.
\end{equation}
For the same reason, we get
\begin{multline} \label{eq:reconstruct_continuous_u}
 u_h(t) = u^{in} - \int_0^t \iint_{\Sp_P} \indic{(0,\Theta_{h}(t))}(s) \Bigg(\V^{+}(u_h(s),P_h^-(p),R_h^-(r)) \\ 
  - \V^{-}(u_h(s),P_h^-(p),R_h^-(r)) \Bigg) f_h(s,p,r) \, drdpds
\end{multline}
In view of \eqref{eq:reconstruct_continuous_f} and \eqref{eq:reconstruct_continuous_u} the conclusion readily follows. 
Indeed, to pass to the limit in \eqref{eq:X}, it is convenient to introduce $\tilde X_h$ where $f_h$ is replaced by $\tilde f_h$, then the same arguments as Section \ref{sec:reconstruct_coagulation} holds true. We write $X_h = X_h -\tilde X_h + \tilde X_h$, then it is clear that 
\[ \norm{X_h -\tilde X_h}_{L^\infty(0,T)} \rightarrow 0\,, \]
by virtue of \eqref{eq:cv_f-tildef}. Then, for all $t\in(0,T)$ we prove that $\tilde X_h$ converge towards the right term by Lemma \ref{lem:cvg}. For \eqref{eq:Y} and \eqref{eq:Z} we do the same decomposition, remarking two points. On one hand, the continuity of $\V$ and the pointwise convergence of $u_h$ allow us to correctly pass to the limit in the positive and negative parts of $\V=\V^+ - \V^⁻$. On the other hand, the time integral is treated thanks to the Lebesgue dominated convergence theorem. Equation \eqref{eq:reconstruct_continuous_u} is treated by the same arguments. Proof of Theorem \ref{thm:convergence} is achieved.

\section{Numerical illustration and long-time behaviour}
\label{sec:illustration}

In this section we choose to illustrate our numerical scheme by simulating a particular example which depicts the asymptotic behaviour of the solution. The simulation of this example seeks to show the typical behaviour of the long-time solution for a wide class of coefficients. 

\subsection{The numerical examples} \label{ssec:exp}

Let us first introduce the coefficients and initial conditions we choose for the simulation. We let $u^{in}=0.9$, $P=1$ and for all $(p,r)\in(0,P)\times(0,1)$

\[f^{in}(p,r) = m \cdot \exp\left( - \frac{(\log(p)+2)^2}{2\cdot0.4^2} - \frac{(r-0.2)^2}{2\cdot0.05^2} \right) \]
with $m$ is a normalisation constant such that 
\[ \int_0^P \int_0^1 r p f^{in}(p,r)\,drdp = 0.1 \,.\]
We notice in $r=0$ and $1$ this function is closed to zero, so numerically we require it vanishes. In this case, $\rho = 1$ since 
\[ u^{in} + \int_0^P \int_0^1 r p f^{in}(p,r)\,drdp  = \rho \,.\]
Then, we used the association-dissociation and the coagulation rates 
\[ \V(u,p,r) = 4p(1-r) u -r \,, \text{ and } \, a(p,r;p',r') = 1\,.\]

In Figures \ref{fig:simu1} and \ref{fig:simu2}, we present the result obtain with the numerical scheme introduced in \ref{sec:scheme}. The first picture of Figure \ref{fig:simu1} is the initial condition. Then, in the second and third, we see that polymers are capturing  metal ions since for each $p$ the distribution shifts towards greater $r$. It is confirmed by Figure \ref{fig:simu_u}, where the concentration $u$ at the same time is decreasing. In the last picture of Figure \ref{fig:simu1}, it appears biggest polymers since the tail of the distribution moves to the right. Finally, in Figure \ref{fig:simu2}, the distribution $f$ seems to be well concentrated onto a curve while the mass moves towards the biggest polymers (right). At this stage, the concentration of metal ions seems to reach a steady state, see Figure \ref{fig:simu_u}, and the coagulation is predominant.

\begin{figure}[!htb]
 {\centering
 \includegraphics{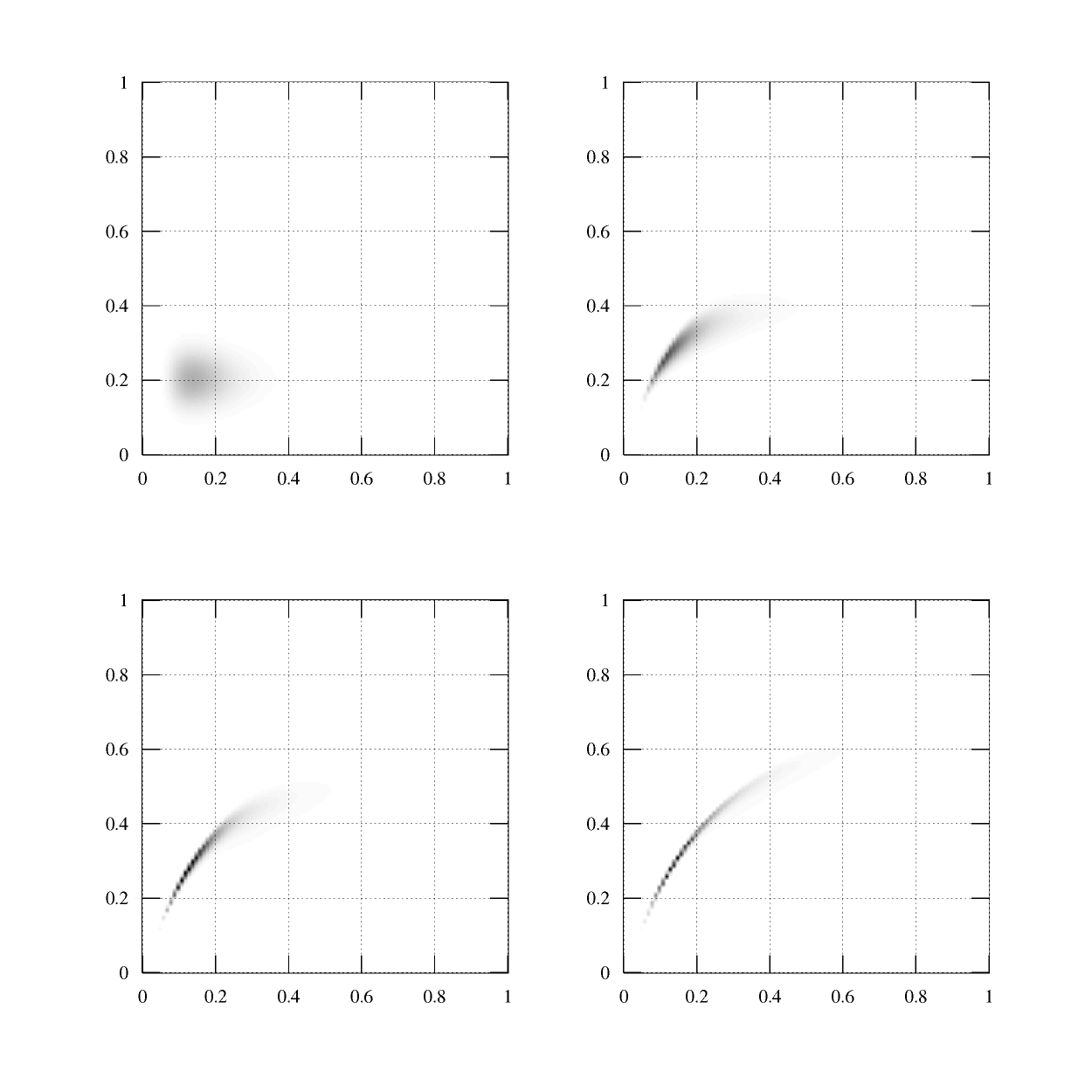} }
 \caption{ {\bf Evolution of the configurational distribution of polymers}. Each snapshot represents the solution at different times from left to right and up to down with $t=0,\, 0.125,\, 0.25,\, 0.5$ where  $p$ is in abscissa and $r$ in ordinate and the grayscale color vary from white to black when $f(t,p,r)$ vary from $0$ to $550$. The simulation was performed with the condition described in Section \ref{sec:illustration} on a regular grid $100\times 100$, {\it i.e.} $\Dp =\Dr = 0.01$ and a time step $\Dt = 1.25 \cdot 10^{-4}$. }
  \label{fig:simu1}
\end{figure}

\begin{figure}[!htb]
 {\centering
 \includegraphics{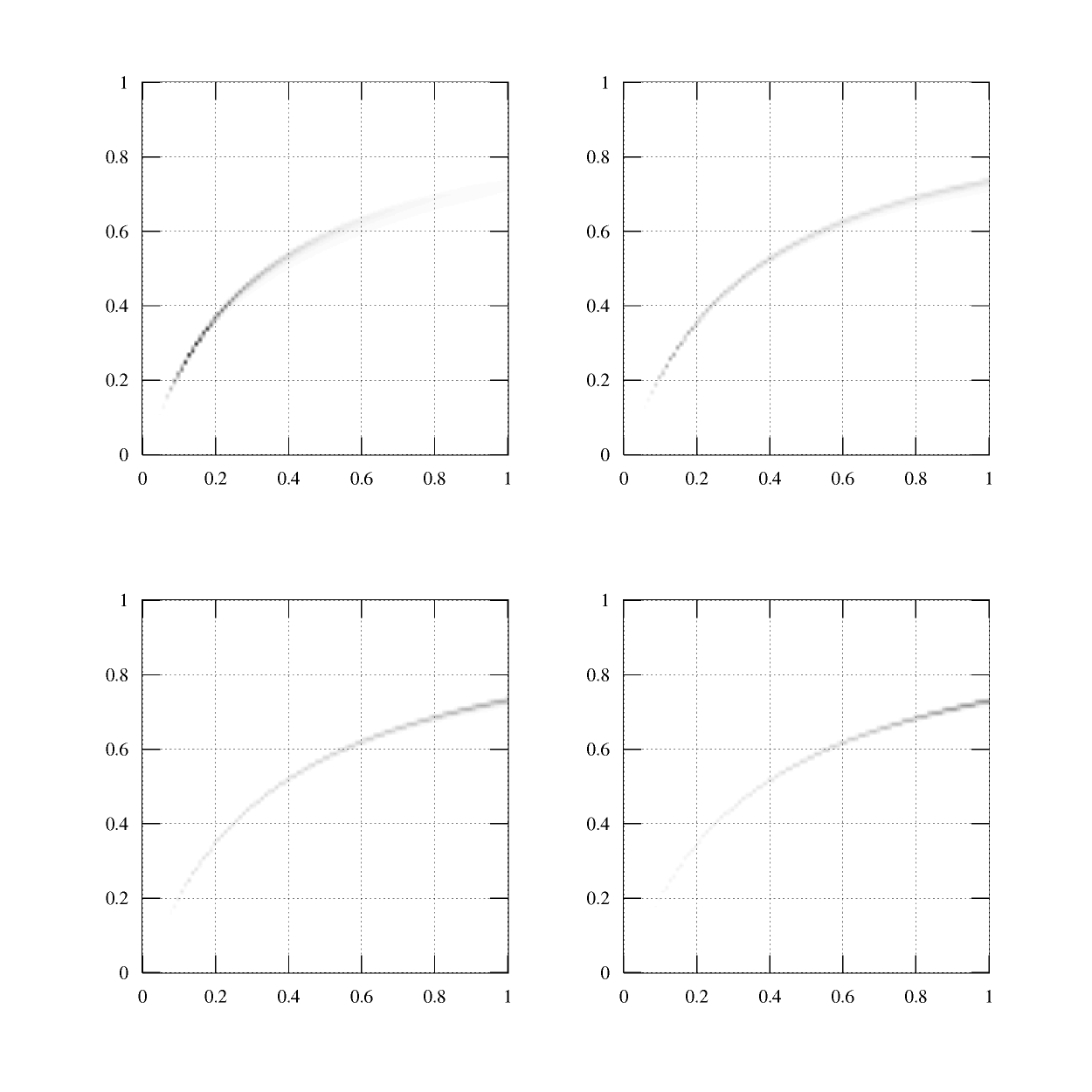} }
 \caption{ {\bf Evolution of the configurational distribution of polymers after a while}. Each snapshot represents the solution at different times from left to right and up to down with $t=1,\, 2, \, 3.5,\, 5$ where  $p$ is in abscissa and $r$ in ordinate the grayscale color vary from white to black when $f(t,p,r)$ vary from $0$ to $350$ . The simulation was performed with the condition described in Section \ref{sec:illustration} on a regular grid $100\times 100$, {\it i.e.} $\Dp =\Dr = 0.01$ and a time step $\Dt =  1.25 \cdot 10^{-4}$.}
  \label{fig:simu2}
\end{figure}
 
\begin{figure}[!htb]
 \centering \includegraphics{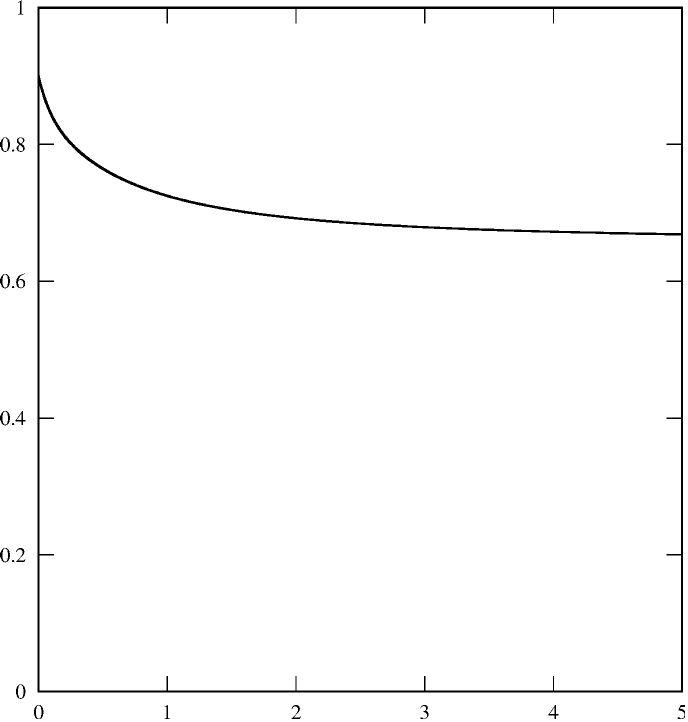} 
 \caption{ {\bf Evolution of the metal ions concentration}. In abscissa is the time $t$ and ordinate is the concentration $u(t)$. The simulation was performed with the condition described in Section \ref{sec:illustration} and corresponds to the solution associate to the numerical solution given in Figures \ref{fig:simu1} and \ref{fig:simu2}. }
 \label{fig:simu_u}
\end{figure}
\vfill\hspace*{0mm}
\subsection{Long-time behaviour}

In the case where the association-dissociation rate has a unique $0$, it is understood that for all $u(t)>0$ and $p>0$, there exists a unique number denoted by $r_t(p)$ such that 

\[\V(u(t),p,r_t(p))=0\, ,\]

then, the distribution function $f$  concentrates towards the cure $p\mapsto r_t(p)$. From a chemical point of view, this curve represents the instantaneous  quantity of metal ion at equilibrium with a polymer of size $p$. It rests on the hypothesis that each size of polymers (under fixed conditions) get a unique preferential ratio of metal ions $r$. The case where  $\V$ has more than one zero would be more complex.

Then, if $u(t)$ converges to a constant $u^\infty$ when $t\rightarrow +\infty$ this curves is given by $r_\infty(p)$. In the numerical example we gave this curves is 

\[ r_\infty (p) = \frac{4p}{4pu^\infty+1}\,. \]

Moreover if we take $u^\infty \simeq u(T=5)$ this curves fit well with what we see in Figure \ref{fig:simu2}. This hypothesis would mean that each size $p$ of polymers has a preferential ratio $r$ of metal ions for given conditions of temperature, pH, \emph{etc}. It would remains to prove that experimentally, if not, and if there more than one nullcline to $\V$ the behaviour would be different and probably dependent on the initial conditions. 

Thus, if we had to conjecture the long-time behaviour of the solution, we would say it is like
\[ f(t,p,r) \sim_{t\rightarrow +\infty} g(t,p) \delta_{r_t(p)}(r)\]
where $g$ satisfies an equation (defined later) and $r_t(p)$ is the solution of $\V(u(t),p,r_t(p))=0$, uniquely define for all $p>0$ and $t\geq0$. It gives the instantaneous equilibrium of the association-dissociation reactions. Now, if we plug $g(t,p) \delta_{r_t(p)}(r)$ (which is a measure) in the weak formulation, we formally get, when taking a test function $\varphi(p,r) = \psi(p)$,

\begin{equation}\label{nonautonomous}
\int_0^\infty g(t,p) \psi(p) \,  dp - \int_S g^{in}(p) \psi(p) \,  dp =  \int_0^t \int_0^\infty \mathcal Q(g,g)(s,p) \psi(p) \, dr dp \, ds\,, 
\end{equation}
with
\begin{multline} \label{eq:Q_bar}
 \int_{0}^{+\infty} \mathcal Q (g,g)(t,p) \psi(p)\, dp = \frac 1 2 \int_{0}^{+\infty} \int_{0}^{+\infty}  b(t;p,p')g(t,p)g(t,p') \\
 \times \Big[ \psi(p+p') - \psi(p) -\psi(p') \Big] dp'dp\,.
\end{multline}
where $b(t;p,p') = a(p,r_t(p);p',r_t(p')$. The operator $\mathcal Q$ in \eqref{eq:Q_bar} is exactly the weak formulation of the {\it classical} coagulation operator, see for instance \cite{Laurencot2002} among others, with a non-autonomous coagulation rate. The non-autonomous coagulation equation received a little attention up to our knowledge except in \cite{Mclaughlin}. 

Thus $g$ would satisfies a non-autonomous coagulation equation \eqref{nonautonomous}. But in the case where $u(t)$ reach a steady state,  $b(t;p,p') \rightarrow b_\infty(p.p')=a(p,r_\infty(p);p',r_\infty(p'))$. 
It might be probable the non-autonomous coagulation equation behave asymptotically like an autonomous coagulation equation with coefficient $b_\infty$. The long-time behaviour of an autonomous coagulation equation reveal self-similarity, see for instance \cite{Fournier2005}. Here, the interpretation we propose from the numerical simulation is a propagation, probably with a self-similar profile, of the distribution over the curve $p\mapsto r_\infty(p)$. The analysis of this problem would be a full work in its own right. So we leave it for now. 

Nevertheless, such behaviour would be taken into account in the experimental procedures. Indeed, the choose of the size distribution could be crucial in the efficacy of the process but they need to diminish the effect of the coagulation to avoid its interference with the membrane for instance \cites{Rivas2003,Rivas2011}.  

\section{Conclusion}

In this work, we dealt with a new model with applications to water-polymers getting particular affinity with metal-ions. These equations can be seen as a variation around the coagulation equation or LS equation. Nevertheless, it includes various specific features which make it an original problem. We want to mention particularly the conservations involved, the nature of the configuration space and the structure of the coagulation operator. We established a first result of existence for a large class of initial data. Then, we established a finite volume scheme and we proved a convergence result. This numerical scheme is used to get an approximation of the solution in particular case which illustrate the long-time behaviour of the solution.

There are several possible extension of this work. In the first section, the class of coefficient should be relax, and may be the monotonicity. To relax the coefficients, it would be possible to use a similar $L^1 - weak$ stability principle as done for the convergence  of the numerical scheme by sequence of approximating coefficients. Of course the question of the uniqueness is still open here. Concerning the numerical scheme, it would be interesting to develop a new one which captures in a better manner the concentration on the curve. Finally, It remains to rigorously demonstrate the question of the long-time behaviour.

\section*{Acknowledgments}
The authors would thanks Bernab\'e Rivas and Julio S\'anchez, from the Department of Polymers at University of Concepci\'on, for their time, advices and very helpful discussions to derive the model.

\noindent EH acknowledges support of FONDECYT Postdoc grant no. 3130318 (Chile).

\noindent MS thanks the support of FONDECYT grant no. 1140676, CONICYT project Anillo ACT1118 (ANANUM), Red Doctoral REDOC.CTA, project UCO1202 at Universidad de Concepci\'on, Basal, CMM, Universidad de Chile and CI\textsuperscript{2}MA, Universidad de Concepci\'on.

\bibliographystyle{amsplain}

\end{document}